\documentclass[reqno]{amsart}

\usepackage{amssymb,amsfonts,amsmath,amsbsy,amscd}
\usepackage{amsthm}
\usepackage{dsfont}

\newcommand{\ol}[1]{\overline{#1}}

\numberwithin{equation}{section}

\newcommand{\R}{\ensuremath{\mathbb{R}}}

\newcommand{\Rn}{\ensuremath{\mathbb{R}^n}}

\newcommand{\N}{\ensuremath{\mathbb{N}}}

\newcommand{\dist}{\operatorname{dist}}
\newcommand{\Id}{\operatorname{Id}}
\newcommand{\sym}{\operatorname{sym}}

\newcommand{\sd}{\, d}

\newcommand{\eps}{\ensuremath{\varepsilon}}
\newcommand{\weight}[1]{\langle #1\rangle}

\newcommand{\Div}{\operatorname{div}}

\newcommand{\Tr}{\operatorname{Tr}}

\newcommand{\E}{\mathbb{E}}

\newcommand{\STC}{\sigma}
\newcommand{\Jump}[1]{[\![ #1]\!]}
\newcommand{\MC}{H}

\newcommand{\sx}{{\sf x}}
\newcommand{\sy}{{\sf y}}
\newcommand{\U}{\mathcal{U}}

\newenvironment{proof*}[1]{{\noindent\bf Proof
#1:}}{\hspace*{\fill}\rule{1.2ex}{1.2ex}\\ }

\numberwithin{equation}{section}

\newtheorem{theorem}{Theorem}[section]
\newtheorem{proposition}[theorem]{Proposition}
\newtheorem{corollary}[theorem]{Corollary}
\newtheorem{lemma}[theorem]{Lemma}

\theoremstyle{definition}

\theoremstyle{definition} 
\newtheorem{remark}[theorem]{Remark}

\DeclareMathOperator{\diver}{div}
\DeclareMathOperator{\dive}{div}

\newcommand{\beq}{\begin{equation}}
\newcommand{\eeq}{\end{equation}}
\newcommand{\bea}{\begin{eqnarray}}
\newcommand{\eea}{\end{eqnarray}}

\newcommand{\cA}{{\mathcal A}}

\newcommand{\cE}{{\mathcal E}}

\def\eps{\varepsilon}

\DeclareMathSymbol{\complement}{\mathord}{AMSa}{"7B}

\def\vv<#1>{\langle #1\rangle}
\def\Vv<#1>{\bigl\langle #1\bigr\rangle}

\begin{document}

\setlength{\arraycolsep}{0.5ex}

\title[Two-Phase Navier-Stokes-Mullins-Sekerka Equations]
{Well-Posedness and Qualitative Behaviour of Solutions for a
Two-Phase Navier-Stokes-Mullins-Sekerka System}

\author[Helmut Abels]{Helmut Abels}
\address{Fakult\"at f\"ur Mathematik \\  Universit\"at Regensburg  \\ Universit\"atsstra\ss e 31\\  D-93053 Regensburg, Germany}
\email{helmut.abels@mathematik.uni-regensburg.de}

\author[Mathias Wilke]{Mathias Wilke}
\address{Institut f\"ur Mathematik  \\
         Martin-Luther-Universit\"at Halle-Witten\-berg\\
         Theodor-Lieser-Str.~5\\
         D-06120 Halle, Germany}
\email{mathias.wilke@mathematik.uni-halle.de}

\date{\today}


\begin{abstract}
We consider a two-phase problem for two incompressible, viscous and immiscible fluids which are separated by a sharp interface. The problem arises as a sharp interface limit of a diffuse interface model. We present results on local existence of strong solutions and on the long-time behavior of solutions which start close to an equilibrium. To be precise, we show that as time tends to infinity, the velocity field converges to zero and the interface converges to a sphere at an exponential rate.
\end{abstract}
\maketitle
{\small\noindent
\vspace{0.1in}\\

 \vspace{0.1in}\\
\date{today}}


\section{Introduction}
We study the flow of two incompressible, viscous and immiscible fluids inside a bounded domain $\Omega\subset \Rn$,
$n=2,3$. The fluids fill domains $\Omega^+(t)$ and $\Omega^-(t)$, $t>0$,
respectively, with a common interface $\Gamma(t)$ between both
fluids. The flow is
described in terms of the velocity $v\colon  (0,\infty)\times \Omega \to \R^n$ and
the pressure $p\colon (0,\infty)\times \Omega  \to \R$ in both fluids in Eulerian coordinates.
We assume the fluids to be of Newtonian type, i.e., the stress tensors of the fluids are of
the form $T (v,p)= 2\mu^\pm Dv -pI$ in $\Omega^\pm(t)$ with constant viscosities $\mu^\pm>0$ and
$2Dv= \nabla v +\nabla v^T$. Moreover, we consider the case with surface
tension at the interface. In this model  the densities of the fluids are assumed to be the same and for simplicity set to one.
For the evolution of the phases we take diffusional
effects into account and consider a contribution to the flux that is
proportional to the negative gradient of the chemical potential $\mu$.
Precise assumptions are made below. This is motivated e.g. from studies of spinodal decomposition in certain polymer mixtures, cf. \cite{SpinodalDecompFluids}.

To formulate our model we introduce some notation first. Denote by $\nu_{\Gamma(t)}$ the unit
normal of $\Gamma(t)$ that points outside $\Omega^+(t)$ and by $V$ and
$H$ the normal velocity and scalar mean curvature of $\Gamma(t)$
with respect to $\nu_{\Gamma(t)}$. By $\Jump{\cdot}$ we denote
the jump of a quantity across the interface in direction of $\nu_{\Gamma(t)}$, i.e.,
\begin{equation*}
  [\![f]\!](x)= \lim_{h\to 0}(f(x+h\nu_{\Gamma(t)})-f(x-h\nu_{\Gamma(t)}))\quad \text{for}\ x\in \Gamma(t).
\end{equation*}
Then our model is described by the following system
\begin{alignat}{2}\label{eq:1}
  \partial_t v + v\cdot \nabla v - \Div T(v,p) &= 0 &\quad & \text{in}\
  \Omega^\pm(t)\ \text{for}\ t>0, \\\label{eq:2}
\Div v &= 0 &\ & \text{in}\ \Omega^\pm(t)\ \text{for}\ t>0, \\\label{eq:2'}
m\Delta \mu &= 0 &\ & \text{in}\ \Omega^\pm(t)\ \text{for}\  t>0, \\\label{eq:3}
 -\nu_{\Gamma(t)}\cdot [\![T(v,p)]\!] &= \STC H \nu_{\Gamma(t)}  &\ & \text{on}\ \Gamma(t)\ \text{for}\  t>0,
  \\\label{eq:4}
V-\nu_{\Gamma(t)}\cdot v|_{\Gamma(t)} &=  - m[\![\nu_{\Gamma(t)}\cdot \nabla\mu ]\!] && \text{on}\ \Gamma(t)\ \text{for}\ t>0,\\\label{eq:4'}
\mu|_{\Gamma(t)} & = \STC H && \text{on}\ \Gamma(t)\ \text{for}\ t>0,
\end{alignat}
together with the initial and boundary conditions
\begin{alignat}{2}\label{eq:5}
v|_{\partial\Omega} &= 0 &\ & \text{on}\ \partial\Omega\ \text{for}\ t>0, \\\label{eq:5'}
 \nu_\Omega\cdot m\nabla  \mu|_{\partial\Omega}  &= 0 &\ & \text{on}\ \partial\Omega\ \text{for}\ t>0,\\
\Omega^+(0)&= \Omega_0^+,&& \label{eq:6'}\\\label{eq:6}
v|_{t=0} &= v_0 &\ & \text{in}\ \Omega,
\end{alignat}
where $v_0, \Omega_0^+$ are given initial data satisfying
$\partial\Omega_0^+\cap \partial\Omega=\emptyset$ and where $\STC,m>0$ are
a surface tension and a mobility constant, respectively. Here and in the following it is
assumed that $v$ and $\mu$ do not jump across $\Gamma(t)$,
i.e.,
$$
[\![v]\!]=[\![\mu]\!]=0\qquad \text{on}\ \Gamma(t)\ \text{for}\  t>0.
$$
Equations (\ref{eq:1})-(\ref{eq:2}) describe the conservation of linear
momentum and mass in both fluids and (\ref{eq:3}) is the balance of forces
at the boundary. 
The equations for $v$ are
complemented by the non-slip condition \eqref{eq:5} at the boundary of
$\Omega$.
The conditions \eqref{eq:2'}, \eqref{eq:5'} describe together with
\eqref{eq:4} a continuity equation for the masses of the phases, and
\eqref{eq:4'} relates the chemical potential $\mu$ to the $L_2$-gradient
of the surface area, which is given by the mean curvature of the
interface.

For $m=0$ the velocity field $v$ is independent of $\mu$. In this case,
(\ref{eq:4}) describes the usual kinematic condition that the
interface is transported by the flow of the surrounding fluids and
(\ref{eq:1})-(\ref{eq:6}) reduces to the classical model of a two-phase
Navier--Stokes flow as for example studied by  Denisova and
Solonnikov~\cite{DenisovaTwoPhase} and K\"ohne et al.~\cite{KPW11},
where short time existence of strong
solutions is shown. 
On the other hand, if $m>0$, the equations (\ref{eq:2'}), (\ref{eq:4'}), (\ref{eq:5'})
with $v=0$ define the Mullins--Sekerka flow of a family of interfaces.
This evolution describes the gradient flow for the surface area functional
with respect to the $H^{-1}(\Omega)$ inner product.
Therefore we will also call (\ref{eq:1})-(\ref{eq:6})
the Navier-Stokes/Mullins-Sekerka system.

The motivation to consider (\ref{eq:1})-(\ref{eq:6}) with $m>0$ is twofold:
First of all, the modified system gives a regularization of the
classical model $m=0$ since the transport equation for the evolution of the interface is replaced by a third order parabolic evolution equation
(cf. also the effect of $m>0$ in \eqref{eq:NSCH3} below).
Secondly, (\ref{eq:1})-(\ref{eq:6}) appears as sharp
interface limit of the following diffuse interface model, introduced by
Hohenberg and Halperin~\cite{HohenbergHalperin} and rigorously derived by
Gurtin et al.~\cite{GurtinTwoPhase}:
\begin{alignat}{2}\label{eq:NSCH1}
  \partial_t v + v\cdot \nabla v - \dive (2\mu (c)Dv) + \nabla p &= -\eps\dive
  (\nabla c \otimes \nabla c)
  &\qquad&\text{in}\ \Omega\times (0,\infty), \\\label{eq:NSCH2}
  \dive v &=0 &\qquad&\text{in}\ \Omega\times (0,\infty), \\\label{eq:NSCH3}
  \partial_t c + v\cdot\nabla c &= m\Delta \mu &\qquad&\text{in}\ \Omega\times (0,\infty), \\\label{eq:NSCH4}
\mu &= \eps^{-1}f'(c) - \eps\Delta c&\qquad& \text{in}\ \Omega\times (0,\infty),\\\label{eq:NSCH5}
  v|_{\partial\Omega} &=0 &\qquad& \text{on}\ \partial\Omega\times (0,\infty),\\\label{eq:NSCH6}
  \partial_n c|_{\partial\Omega} = \partial_n \mu|_{\partial\Omega} &= 0
  &\qquad& \text{on}\ \partial\Omega\times (0,\infty),\\\label{eq:NSCH7}
  (v,c)|_{t=0} &= (v_0,c_0) &\qquad& \text{in}\ \Omega.
\end{alignat}
Here $c$ is the concentration of one of the fluids, where we note that a
partial mixing of both fluids is assumed in the model, and $f$ is a suitable
``double-well potential'' e.g. $f(c)=c^2(1-c)^2$. Moreover,
$\eps>0$ is a small parameter related to the interface thickness, $\mu$
is the so-called chemical potential and $m>0$ is the mobility. We refer to
\cite{ModelH,BoyerModelH} for some analytic results for this model and to \cite{NonNewtonianModelH,GrasselliPrazak} for results for a non-Newtonian variant of this model.
For some results on the sharp interface limit  of \eqref{eq:NSCH1}-\eqref{eq:NSCH7} we refer to A. and R\"oger~\cite[Appendix]{NSMS} and A., Garcke, and Gr\"un~\cite{AbelsGarckeGruen2}.

The purpose of this paper is to prove existence of strong solutions of \eqref{eq:1}-\eqref{eq:6} locally in time. Moreover, we will prove stability of  spheres, which are equilibria for the systems. (More precisely, we show dynamic stability of the solutions $v\equiv 0$, $\mu,p \equiv \text{const.}$, and $\Omega^+(t)= B_R(x)\subset \Omega$ for all $t>0$.)
Existence of weak solutions for large times and general initial data was shown in \cite{NSMS}.

In the following we will assume that $\Omega\subset \Rn$, $n=2,3$, is a bounded domain with $C^4$-boundary and that $\mu^\pm, m,\sigma>0$ are constants. One essential feature of \eqref{eq:1}-\eqref{eq:6} is the coupling of lower order between the velocity field $v$ and the chemical potential $\mu$ in equation \eqref{eq:4}. Indeed, we will obtain functions in the regularity classes $\mu\in L_p(J;W_p^2(\Omega\backslash\Gamma(\cdot)))$ and
$$v\in H_2^1(J;L_2(\Omega)^n)\cap L_2(J;H_2^2(\Omega\backslash\Gamma(\cdot))^n).$$
Taking the trace to $\Gamma(t)$ yields $\nabla\mu|_{\Gamma}\in L_p(J;W_p^{1-1/p}(\Gamma(\cdot))^n)$ and by complex interpolation and Sobolev embeddings we obtain
$$v\in H_2^1(J;L_2(\Omega)^n)\cap L_2(J;H_2^2(\Omega\backslash\Gamma(\cdot))^n)\hookrightarrow L_q(J;W_p^{1}(\Omega\backslash\Gamma(\cdot))^n),$$
where $q>p$ and $p\le 2(n+2)/n$. This shows that the trace
$$v|_{\Gamma}\in L_q(J;W_p^{1-1/p}(\Gamma(\cdot))^n)$$
possesses more regularity with respect to time compared to $\nabla\mu|_{\Gamma}$. We make essential use of this fact by applying the following strategy for the proof of local-in-time well-posedness. After parameterizing the free interface $\Gamma(t)$ via the Hanzawa transform by a height function $h$, the basic idea is to reduce \eqref{eq:1}-\eqref{eq:6} to a single equation for $h$. To this end we first assume that the interface, hence $h$, is given. Then we solve the (transformed) two-phase Navier-Stokes equations to obtain a solution operator $v=S_{NS}(h)$. Doing the same for the (transformed) two-phase Mullins-Sekerka equations, this yields a solution operator $\mu=S_{MS}(h)$. Finally, we consider the transformed evolution equation \eqref{eq:4} for the height function $h$ and replace $v$ and $\mu$ by $S_{NS}(h)$ and $S_{MS}(h)$, respectively, to obtain a single equation for $h$. This quasilinear parabolic equation in turn can be solved by parabolic theory. The only point one has to take care of is that the solution operator $S_{NS}$ in nonlocal in time and space. Therefore one has to deal with a parabolic equation with local leading part and lower order perturbations which are nonlocal (in time and space). Having solved the single equation for $h$ readily computes the velocity, the pressure and the chemical potential by the solution operators obtained before.

Let us comment on the choice of an $L^2$-setting for the Navier-Stokes part, while the equations for the height function $h$ and the chemical potential $\mu$ are treated by an $L^p$-theory, $p>2$. One advantage is that the optimal regularity result for the two-phase Navier-Stokes equations with a given interface (see Theorem \ref{thm:MaxReg}) is more or less easy to prove since it relies solely on resolvent estimates in $L^2$. Another benefit is the reduction of the regularity of the initial velocity and the compatibility conditions at $t=0$. For instance, if $p=2$, then there is no compatibility condition for the initial value $v_0$ coming from the jump of the stress tensor, that is equation \eqref{eq:3}.

The structure of the paper is as follows: First we introduce some basic notation and auxiliary results in Section~\ref{sec:Prelim}. Then we will prove that for a given sufficiently smooth interface $\Gamma(t)$ the Navier-Stokes part of the system, i.e., \eqref{eq:1}-\eqref{eq:2}, \eqref{eq:3}, \eqref{eq:5}, \eqref{eq:6} possesses for sufficiently small times a unique strong solution $v$ in $L^2$-Sobolev spaces, which are second order in space and first order in time. This result is proved using a coordinate transformation to the initial domains $\Omega_0^\pm$ which goes back to Hanzawa and applying the contraction mapping principle. A key tool in our analysis will be a maximal $L^2$-regularity result for the linearized Stokes system, which is proved in the appendix. Afterwards in Section~\ref{sec:local} we prove that the full system possesses a strong solution locally in time for sufficiently smooth initial data by reducing the whole system to a single equation for the height function $h$ (see above). Then in Section~\ref{sec:Global} we prove stability of the stationary solutions that are given by $v\equiv 0$, $\mu,p\equiv const.$ and $\Gamma(t)\equiv \partial B_r(x_0)\subset \Omega$ and we show that $(v(t),\Gamma(t))$ converges to an equilibrium as $t\to\infty$ at an exponential rate.


\section{Preliminaries}\label{sec:Prelim}
\subsection{Notation and Function Spaces}

If $X$ is a Banach space, $r>0$, $x\in X$, then $B_X(x,r)$ denotes the (open) ball in $X$ around $x$ with radius $r$. We will often write simply $B(x,r)$ instead of $B_X(x,r)$ if $X$ is well known from the context.

The usual $L^p$-Sobolev spaces are denoted by $W^k_p(\Omega)$ for $k\in\N_0,1\leq p\leq \infty$, and $H^k(\Omega)= W^k_2(\Omega)$. Moreover $W^k_{p,0}(\Omega)$ and $H^k_0(\Omega)$ denote the closure of $C_0^\infty(\Omega)$ in $W^k_p(\Omega)$, $H^k(\Omega)$, respectively. The vector-valued variants are denoted by $W^k_p(\Omega;X)$ and $H^k(\Omega;X)$, where $X$ is a Banach space. 
The usual Besov spaces are denoted by $B^s_{p,q}(\Rn)$, $s\in \R$, $1\leq p,q\leq \infty$, cf. e.g. \cite{Interpolation,Triebel1}. If $\Omega\subseteq\Rn$ is a domain, $B^s_{p,q}(\Omega)$ is defined by restriction of the elements of $B^s_{p,q}(\Rn)$ to $\Omega$, equipped with the quotient norm. We refer to \cite{Interpolation,Triebel1} for the standard results on interpolation of Besov spaces and Sobolev embeddings. We only note that $B^s_{p,q}(\Omega)$ and $W^k_p(\Omega)$ are retracts of $B^s_{p,q}(\Rn)$ and $W^k_p(\Rn)$, respectively, because of the extension operator constructed in Stein~\cite[Chapter VI, Section 3.2]{Stein:SingInt} for bounded Lipschitz domains. In particular, we have
\begin{equation}\label{eq:InterpolationSobolevBesov}
(W^k_{p_0}(\Omega), W^{k+1}_{p_1}(\Omega))_{\theta,p}= B^{k+\theta}_{p,p}(\Omega)\qquad \text{if}\ \frac1p=\frac{1-\theta}{p_0}+\frac{\theta}{p_1},k\in \N_0,
\end{equation}
for all $\theta\in (0,1)$, cf. \cite[Section~2.4.2 Theorem~1]{Triebel1}. We also denote $W^{k+\theta}_p(\Omega)=B^{k+\theta}_{p,p}(\Omega)$ for $k\in\N_0$, $\theta\in (0,1)$, $1\leq p\leq \infty$.
Furthermore, we define
\begin{eqnarray*}
  L^2_{(0)}(\Omega) &=& \left\{f\in L^2(\Omega): \int_{\Omega} f(x) \sd x=0\right\},\\
  L^2_\sigma(\Omega)&=&\overline{\left\{f\in C^\infty_0(\Omega)^n:\Div f =0\right\}}^{L^2(\Omega)^n}.
\end{eqnarray*}

In order to derive some suitable estimates  we will use vector-valued Besov spaces $B^s_{q,\infty}(I;X)$,
where $s\in (0,1)$, $1\leq q\leq\infty$, $I$ is an interval, and $X$ is a Banach space. They are defined as
\begin{eqnarray*}
  B^s_{q,\infty}(I;X)&=&\left\{f\in L^q(I;X): \|f\|_{B^s_{q,\infty}(I;X)}<\infty\right\},\\
    \|f\|_{B^s_{q,\infty}(I;X)} &=&\|f\|_{L^q(I;X)}+
    \sup_{0<h\leq 1}\|\Delta_h f(t)\|_{L^q(I_h;X)},
\end{eqnarray*}
where $\Delta_h f(t)= f(t+h)-f(t)$ and $I_h=\{t\in I: t+h\in I\}$.
Moreover, we set $C^s(I;X)= B^s_{\infty,\infty}(I;X)$, $s\in (0,1)$.
Now let $X_0,X_1$ be two Banach spaces. Using
$f(t) - f(s) = \int_s^t \frac{d}{dt} f(\tau)\sd \tau$
it is easy to show that for $1\leq q_0 <q_1\leq \infty$
\begin{equation}\label{eq:BesovEmbedding}
  W^1_{q_1}(I;X_1)\cap L^{q_0}(I;X_0) \hookrightarrow B^\theta_{q,\infty}(I;X_\theta),\qquad \frac1q = \frac{1-\theta}{q_0}+
\frac{\theta}{q_1},
\end{equation}
where $\theta \in (0,1)$ and $X_\theta = (X_0,X_1)_{[\theta]}$ or $X_\theta =(X_0,X_1)_{\theta,r}$, $1\leq r\leq \infty$. Furthermore,
\begin{equation}\label{eq:BesovSobolev}
  B^\theta_{q,\infty}(I;X)\hookrightarrow C^{\theta-\frac1q}(I;X)
\quad \text{for all}\ 0<\theta<1,1\leq q\leq \infty\ \text{with}\ \theta-\frac1q>0,
\end{equation}
cf. e.g. \cite{SimonBesovEmbeddings}.
 Furthermore, for $s\in (0,1)$ we define $ H^s(0,T;X)= B^s_{2,2}(0,T;X)$, where $f\in B^s_{2,2}(0,T;X)$ if and only if $f\in L^2(0,T;X)$ and
 \begin{equation*}
   \|f\|_{B^s_{2,2}(0,T;X)}^2= \|f\|_{L^2(0,T;X)}^2 + \int_0^T\int_0^T\frac{\|f(t)-f(\tau)\|_{X}^2}{|t-\tau|^{2s+1}}\sd t\sd \tau <\infty.
 \end{equation*}
In the following we will use that
\begin{eqnarray*}
  \lefteqn{\int_0^T\int_0^T\frac{\|f(t)-f(\tau)\|_{X}^2}{|t-\tau|^{2s+1}}\sd t\sd \tau}\\
&\leq& \int_0^T\int_0^T|t-\tau|^{2(s'-s)-1}\sd t\sd \tau \|f\|_{C^{s'}([0,T];X)}^2
\leq C_{s',s}T^{2(s'-s)+1} \|f\|_{C^{s'}([0,T];X)}^2
\end{eqnarray*}
for all $0<s<s'\leq 1$, which implies
\begin{equation}\label{eq:VecHEstim}
  \|f\|_{H^{s}(0,T;X)}\leq C_{s,s'}T^{\frac12}\|f\|_{C^{s'}([0,T];X)}\quad \text{for all}\ f\in C^{s'}([0,T];X)
\end{equation}
provided that $0<s<s'\leq 1$, $0<T\leq 1$.

Furthermore, we note that the space of bounded $k$-times continuously differentiable functions $f\colon U\subset X\to Y$ with bounded derivatives are denoted by $BC^k(U;Y)$, where $X,Y$ are Banach spaces and $U$ is an open set. Moreover, $f\in C^k(U;Y)$ if for every $x\in U$ there is some neighborhood $V$ of $x$ such that $f|_{V}\in BC^k(V;Y)$.

We will frequently use the following multiplication result for Besov spaces:
\begin{equation}\label{eq:BesovProd}
  \|fg\|_{B^s_{p, \max(q_1,q_2)}}\leq C_{r,s,p,q}\|f\|_{B^r_{p_1,q_1}}\|g\|_{B^s_{p,q_2}}
\end{equation} 
for all $ f\in B^r_{p_1,q_1}(\Rn),g\in B^s_{p,q_2}(\Rn)$ provided that  $1\leq p\leq p_1\leq \infty$, $1\leq q_1,q_2\leq \infty$, $r>\frac{n}{p_1}$, and
$$-r+n\left(\tfrac{1}{p_1}+\tfrac{1}{p}-1\right)_+ < s\leq r,$$
cf. \cite[Theorem~6.6]{JohnsenPointwiseMultipliers}. Since $W^s_p(\R^{n})= B^s_{p,p}(\R^{n})$ for every $s\in (0,\infty)\setminus \N$,
 this implies that
\begin{equation}\label{eq:AlgebraEstim}
  \|fg\|_{W^s_{p}(\R^n)}\leq C_{s,p} \|f\|_{W^s_{p}(\R^n)} \|g\|_{W^s_{p}(\R^n)}\qquad \text{for all}\ f,g\in W^s_p(\Rn)
\end{equation}
provided that $s-\frac{n}p>0$, $1\leq p\leq \infty$. Concerning composition operators, we note that
\begin{equation}\label{eq:CompOp}
G(f)\in B^s_{p,q}(\Rn)\qquad \text{for all}\ G\in C^\infty(\R)\ \text{with}\ G(0)=0,f\in B^s_{p,q}(\Rn)
\end{equation}
 provided that again $s-\frac{n}p>0$, $1\leq p,q\leq \infty$.
This implies that $f^{-1}\in B^s_{p,q}(\Omega)$ for all $f\in  B^s_{p,q}(\Omega)$ such that $|f|\geq c_0>0$ if $\Omega$ is a bounded Lipschitz domain. Moreover, the mappings $f\mapsto G(f)$ is bounded on $B^s_{p,q}(\R^n)$ under the previous conditions. We refer to Runst \cite{RunstCompOp} for an overview, further results, and references. Furthermore, using the boundedness of $f\mapsto G(f)$ one can easily derive that
\begin{equation*}
  G(\cdot) \in C^1(B^s_{p,q}(\Rn))
\end{equation*}
for any $G\in C^\infty(\R)$ with $G(0)=0$.
To this end one uses
\begin{equation*}
  G(f(x)+h(x))= G(f(x))+ G'(f(x))+\int_0^1 G''(f(x)+th(x))\sd t\, h(x)^2
\end{equation*}
together with \eqref{eq:AlgebraEstim} and the fact that $(G''(f+th))_{t\in [0,1]}$ is bounded in $B^s_{p,q}(\Rn)$.

Finally, by standard methods these results directly carry over to $W^s_p(\Sigma), B^s_{p,q}(\Sigma)$ if $\Sigma$ is an $n$-dimensional smooth compact manifold. Then $G(0)=0$ is no longer required since constant functions are in $B^s_{p,q}(\Sigma)$.

\medskip

\subsection{Coordinate Transformation and Linearized Curvature Operator}
~\\
In the following let $\Sigma\subset \Omega$ be a smooth, oriented, compact and $(n-1)$-dimensional (reference) manifold with normal vector field $\nu_\Sigma$. Moreover, for a given measurable ``height function'' $h\colon \Sigma\to \R$ let
\begin{equation*}
  \theta_h\colon \Sigma \to \Rn\colon x\mapsto x+h(x)\nu_\Sigma(x).
\end{equation*}
Then $\theta_h$ is injective provided that $\|h\|_{L^\infty}\leq a$ for some sufficiently small $a>0$, where $a$ depends on the maximal curvature of $\Sigma$. Moreover, we choose $a$ so small that $3a <\dist(\Sigma,\partial\Omega)$. Then the so-called \emph{Hanzawa transformation} is defined as
\begin{equation}\label{eq:Hansawa}
  \Theta_h(x,t)= x+\chi(d_{\Sigma}(x)/4a)h(t,\Pi(x))\nu_\Sigma(\Pi(x)),
\end{equation}
where $d_\Sigma$ is the signed distance function with respect to $\Sigma$, $\Pi(x)$ is the orthogonal projection onto $\Sigma$, $\chi\in C^\infty(\R)$ such that $\chi(s)=1$ for $|s|<\frac13$ and $\chi(s)=0$ for $|s|>\frac23$ as well as $|\chi'(s)|\leq 4$ for all $s\in\R$, and $\|h\|_{L^\infty}<a$.
It is well-known that $\Theta_h(.,t)\colon\Omega\to\Omega$ is a $C^1$-diffeomorphism. Hence $\Gamma_h:=\Theta_h(\Sigma)=\theta_h(\Sigma)$ is an oriented, compact $C^k$-manifold if $h\in C^k(\Sigma)$ with $\|h\|_{L^\infty(\Sigma)}<a$.

For the following let
\begin{eqnarray}\label{eq:U}
  \U&=& \left\{h\in W^{4-\frac4p}_p(\Sigma): \|h\|_{L^\infty} <a\right\},\\\nonumber
  \mathbb{E}_{1,T}&=& L^p(0,T;W^{4-\frac1p}_p(\Sigma))\cap W^1_p(0,T;W^{1-\frac1p}_p(\Sigma))
\end{eqnarray}
where $3< p \leq \frac{2(n+2)}n$, $0<T<\infty$.
Furthermore, let
\begin{equation}\label{eq:Curvature}
  K(h):= H_{h}\circ \theta_h,
\end{equation}
where $H_h\colon \Gamma_h\to \R$ denotes the mean curvature of $\Gamma_h= \theta_h(\Sigma)$, i.e., it is the sum of all principal curvatures.
\begin{lemma}\label{lem:LCO1}
  Let $3< p\leq \frac{2(n+2)}n$ and $\U\subset W^{4-\frac4p}_p(\Sigma)$ be as above. Then there are functions
  \begin{alignat*}{1}
    P&\in C^1(\U,\mathcal{L}(W^{4-\frac1p}_p(\Sigma),W^{2-\frac1p}_p(\Sigma))),
  \quad  Q\in C^1(\U,W^{2-\frac1p}_p(\Sigma))
  \end{alignat*}
  such that
  \begin{equation*}
    K(\rho)= P(\rho)\rho +Q(\rho)\qquad \text{for all}\ \rho \in \U\cap W^{4-\frac1p}_p(\Sigma).
  \end{equation*}
  Moreover, if $\Sigma=S_R:=\partial B_R(0)$, then
  \begin{equation}\label{eq:DK0}
    DK(0)= D:=D_{S_R}:=-\frac1{n-1} \left(\frac{n-1}{R^2} +\Delta_{S_R}\right).
  \end{equation}
\end{lemma}
\begin{proof}
  The proof follows essentially from the proof of \cite[Lemma~3.1]{ES98} and \cite[Remark~3.2 a.]{ES98}. To this end
  let $\{(U_l,\varphi_l): 1\leq l\leq L\}$ be a localization system for $\Sigma$, i.e., $\Sigma=\bigcup_{l=1}^L U_l$ and $\varphi_l\colon (-a,a)^{n-1}\to U_l$ is a smooth local parametrization of $U_l$ for all $l=1,\ldots,L$. Moreover, let $s=(s_1,\ldots, s_{n-1})$ be the local coordinates of $U_l$ with respect to this parametrization and
  \begin{equation*}
    \rho_l(s):= \rho(\varphi_l(s)),\qquad X_l(s,r):= X(\varphi_l(s),r),\qquad (s,r)\in (-a,a)^n
  \end{equation*}
be the local representations of $\rho,X$, where $X\colon \Sigma\times (-a,a)\to\Rn$ with $X(s,r)=s+r\nu_{\Sigma}(s)$ and $\rho\in  U\subset W^{4-\frac4p}_p(\Sigma)$. Then it follows from \cite[Equations (3.4), (3.5), Remark~3.2 a.]{ES98} that $K(\rho)=P(\rho)\rho +Q(\rho)$, where $P(\rho),Q(\rho)$ have the local representations
\begin{equation*}
  P_l(\rho)= \frac{1}{n-1}\left(\sum_{j,k=1}^{n-1}p_{jk}(\rho)\partial_{s_j}\partial_{s_k} + \sum_{i=1}^{n-1} p_i(\rho)\partial_{s_i}\right),\quad
Q_l(\rho)= \frac{1}{n-1} q(\rho),
\end{equation*}
where
\begin{eqnarray*}
p_{jk}(\rho) &=& \frac1{l^3_\rho}\left(-l^2_\rho w^{jk}(\rho)+ \sum_{l,m=1}^{n-1}w^{jl}(\rho)w^{km}(\rho)\partial_{s_l}\rho \partial_{s_m}\rho \right)\\
  p_{i}(\rho) &=& \frac1{l^3_\rho}\left(l^2_\rho \sum_{j,k=1}^{n-1}w^{jk}\Gamma_{jk}^i+ \sum_{j,l=1}^{n-1} w^{jl}w^{ki}\Gamma_{jk}^n\partial_{s_l}\rho+ \sum_{k,m=1}^{n-1}2 w^{km}\Gamma_{nk}^i\partial_{s_m}\rho\right.\\
&&\left. -\sum_{j,k,l,m=1}^{n-1}w^{jl}w^{km}\Gamma_{jk}^i\partial_{s_l}\rho\partial_{s_m}\rho\right),\\
q(\rho)&=& -\frac1{l_\rho} \sum_{j,k=1}^{n-1} w^{jk}(\rho)\Gamma_{jk}^n(\rho),\quad
l_\rho = \sqrt{1+ \sum_{j,k=1}^{n-1}w^{jk}(\rho)\partial_{s_j}\rho\partial_{s_k}\rho},\\
\Gamma^i_{jk}(\rho)&=& \sum_{m=1}^{n-1}w^{im}(\rho)\partial_{s_j}\partial_{s_k} X\cdot \partial_{s_m} X|_{(s,\rho(s))},\quad i\neq n,\\
\Gamma^n_{jk}(\rho)&=& \partial_{s_j}\partial_{s_k} X\cdot \partial_{s_n} X|_{(s,\rho(s))},\quad
w_{jk}(\rho)(s)= \partial_{s_j} X\cdot \partial_{s_k} X|_{(s,\rho(s))}
\end{eqnarray*}
and $(w^{jk}(\rho)(s))_{j,k=1}^{n-1}$ is the inverse of $(w_{jk}(\rho)(s))_{j,k=1}^{n-1}$.

Since $\Sigma$ is smooth, $X$ and $\partial_{s_j} X\cdot \partial_{s_j}X$ are smooth. Therefore
$w_{jk}(\rho)\in W^{4-\frac4p}_p(\Sigma)$ because of \eqref{eq:CompOp}. Since $\det ((w_{jk})_{j,k=1}^{n-1})\geq c_0 >0$ by construction, we obtain $w^{jk}(\rho)\in W^{4-\frac4p}_p(\Sigma)$ for all $j,k=1,\ldots,n-1$ because of \eqref{eq:CompOp}.

 Moreover, $\partial_{s_j}\rho \in W^{3-\frac4p}_p(\Sigma)$ and therefore
\begin{equation*}
   \sum_{j,k=1}^{n-1}w^{jk}(\rho)\partial_{s_j}\rho\partial_{s_k}\rho\in W^{3-\frac4p}_p(\Sigma)
\end{equation*}
due to \eqref{eq:AlgebraEstim}. Using \eqref{eq:CompOp} again, we obtain  $l_\rho\in  W^{3-\frac4p}_p(\Sigma)$. Proceeding this way, we finally obtain that $p_{jk}(\rho),p_i(\rho), q(\rho)\in W^{3-\frac4p}_p(\Sigma)$ for all $\rho\in \U$. Now \eqref{eq:BesovProd} implies that
\begin{equation*}
  \|au\|_{W^{2-\frac1p}_p(\Sigma)}\leq C_p \|a\|_{W^{3-\frac4p}_p(\Sigma)} \|u\|_{W^{2-\frac1p}_p(\Sigma)}
\end{equation*}
for all $a\in W^{3-\frac4p}_p(\Sigma), u\in W^{2-\frac1p}_p(\Sigma)$. Hence
\begin{alignat*}{1}
  P&\in C^1(\U,\mathcal{L}(W^{4-\frac1p}_p(\Sigma), W^{2-\frac1p}_p(\Sigma)),\\
Q&\in C^1(\U,\mathcal{L}(W^{2-\frac1p}_p(\Sigma))
\end{alignat*}
since the operators are compositions of $C^1$-mappings. Moreover, \eqref{eq:DK0} follows directly from the observations in the proof of \cite[Lemma~3.1]{ES98}.
\end{proof}
\begin{corollary}
  \label{cor:Curvature}
  Let $K$ be as in \eqref{eq:Curvature}. Then
  \begin{equation*}
    K\in C^1(\mathbb{E}_{1,T}\cap \U; H^{\frac14}(0,T;L_2(\Sigma))\cap L_2(0,T;H^{\frac12}(\Sigma))).
  \end{equation*}
  Moreover, for every $\eps>0, 0<T_0<\infty$ there is some $C>0$ such that
  \begin{equation*}
    \|K\|_{BC^1(\mathbb{E}_{1,T}\cap \U_\eps;H^{\frac14}(0,T;L_2(\Sigma))\cap L_2(0,T;H^{\frac12}(\Sigma))) }\leq C
  \end{equation*}
  for all $0<T\leq T_0$, where $\U_\eps =\{a\in \U:\|a\|_{L^\infty(\Sigma)}\leq a-\eps\}$. 
\end{corollary}
\begin{proof}
  We use that
  \begin{equation*}
    K(h)= \sum_{|\alpha|\leq 2} a_\alpha(x,h,\nabla_s h)\partial_s^\alpha h
  \end{equation*}
  for all $h\in C^2(\Sigma)$, where $a_\alpha\colon \Sigma\times \R\times \R^{n-1}\to \R$ is smooth. Since
  \begin{equation*}
    \mathbb{E}_{1,T} \hookrightarrow B^{\frac23}_{p,\infty}(0,T;W^{2-\frac1p}_p(\Sigma))\cap B^{\frac13}_{p,\infty}(0,T;W^{3-\frac1p}_p(\Sigma))
  \end{equation*}
  due to \eqref{eq:BesovEmbedding}
  and
  \begin{equation*}
    B^{\frac23}_{p,\infty}(0,T;W^{2-\frac1p}_p(\Sigma))\hookrightarrow C^{\frac13}([0,T];C^0(\Sigma))
  \end{equation*}
due to \eqref{eq:BesovSobolev} and $p>3$,
  we conclude that
  \begin{equation*}
    a_\alpha(x,h,\nabla_s h)\in C^{\frac13}([0,T];C^0(\Sigma))\quad\text{for all}\ h\in  \mathbb{E}_{1,T}\cap \U
  \end{equation*}
  and for all $|\alpha|\leq 2$.
  Moreover, the mapping
  \begin{equation*}
   \U\cap \mathbb{E}_{1,T}\ni h\mapsto  a_\alpha(x,h,\nabla_s h)\in C^{\frac13}([0,T];C^0(\Sigma))
  \end{equation*}
  is $C^1$ since $a_\alpha$ are smooth. Furthermore, we conclude that
  \begin{eqnarray*}
   \lefteqn{\|a_\alpha(x,h,\nabla_s h)\partial_s^\alpha v\|_{H^{\frac14}(0,T;L_2(\Sigma))}}\\
&\leq & C_\eps\|a_\alpha(x,h,\nabla_s h)\partial_s^\alpha v\|_{B^{\frac13}_{p,\infty}(0,T;L_p(\Sigma))}\\
&\leq & C_\eps\|a_\alpha(x,h,\nabla_s h)\|_{C^{\frac13}([0,T];C^0(\Sigma))}\| v\|_{B^{\frac13}_{p,\infty}(0,T;W^{1-\frac1p}_p(\Sigma))}\\
&\leq & C_\eps\|a_\alpha(x,h,\nabla_s h)\|_{C^{\frac13}([0,T];C^0(\Sigma))}\| v\|_{\mathbb{E}_{1,T}}
  \end{eqnarray*}
for all $|\alpha|\leq 2$, $v\in \mathbb{E}_{1,T}$, $h\in \mathbb{E}_{1,T}\cap \U_\eps$, $\eps>0$.
Since multiplication is smooth (if bounded), it follows that
\begin{equation*}
     K\in BC^1(\mathbb{E}_{1,T}\cap \U_\eps; H^{\frac14}(0,T;L_2(\Sigma)))
\end{equation*}
for any $\eps>0$.
Finally, we use that $a_\alpha(x,h,\nabla_s h)\in BUC([0,T];C^1(\Sigma))$ and
\begin{equation*}
\mathbb{E}_{1,T}\cap \U_\eps \ni h\mapsto a_\alpha(x,h,\nabla_s h)\in BUC([0,T];C^1(\Sigma))
\end{equation*}
is in $C^1$ with bounded derivative. Hence
\begin{equation*}
  a_\alpha(x,h,\nabla_s h)\nabla_s^\alpha h\in L_p(0,T;W^{1-\frac1p}_p(\Sigma))\hookrightarrow L_2(0,T;H^{\frac12}(\Sigma))
\end{equation*}
for every $h\in \U_\eps\cap \mathbb{E}_{1,T}$, $\eps>0$ and the mapping $h\mapsto K(h)$ is in $BC^1$ with respect to the corresponding spaces. Altogether we have proved the corollary.
\end{proof}

\section{Two-Phase Navier-Stokes System for given Interface}

In this section we assume that the family of interfaces $\{\Gamma(t)\}_{t>0}$ is known and we will solve the system \eqref{eq:1}, \eqref{eq:2}, \eqref{eq:3}, \eqref{eq:5}, \eqref{eq:6} together with the jump condition $\Jump{v}=0$.

For the following let $\Sigma\subseteq\Omega$ be a smooth compact $(n-1)$-dimensional reference manifold as in the previous section. Moreover, we assume that there is a domain $\widetilde{\Omega}_0^+\subset\subset \Omega$ such that $\Sigma=\partial\widetilde{\Omega}_0^+$. Moreover, we assume that
\begin{equation*}
  \Gamma(t)= \{x+h(t,x)\nu_{\Sigma}(x):x\in\Sigma\}=: \Gamma_{h(t)}
\end{equation*}
for some $h\in \U\cap \mathbb{E}_{1,T}$, where
\begin{equation*}
\mathbb{E}_{1,T}:=W^1_p(J; X_0)\cap L_p(J,X_1),
\end{equation*}
 $J=[0,T]$, and
\begin{equation*}
  X_0 = W^{1-\frac1p}_p(\Sigma),\qquad X_1= W^{4-\frac1p}_p(\Sigma)
\end{equation*}
for $p>\max(\tfrac{n+3}2,3)=3$, $n=2,3$, and $\nu_{\Sigma}(x)$ is the exterior normal on $\partial\widetilde{\Omega}_0^+= \Sigma$. Here $\U$ is as in \eqref{eq:U}.

For given $h\in \mathbb{E}_{1,T}$ let $\tilde{h}=Eh\in \widetilde{\E}_{1,T}$, where
$$
E\colon {\E}_{1,T}\to \widetilde{\E}_{1,T}:=W^1_p(J;W^1_p(\Sigma_a))\cap L_p(J; W^4_p(\Sigma_a))
$$
is a continuous extension operator and $\Sigma_a=\{x\in \Omega:\dist(x,\Sigma)< a\}$.
Then by Lion's trace method of real interpolation, we have
\begin{eqnarray}\label{eq:BUCEmbedding}
  \widetilde{\mathbb{E}}_{1,T}\hookrightarrow BUC([0,T]; \widetilde{X}_\gamma),\qquad \widetilde{X}_\gamma= W^{4-\frac3p}_p(\Sigma_a)\hookrightarrow C^2(\ol{\Sigma_a})
\end{eqnarray}
since $p>\frac{n+3}2$.
Moreover, if we equip $\mathbb{E}_{1,T}$ and $\widetilde{\mathbb{E}}_{1,T}$ with the norms
\begin{eqnarray*}
  \|u\|_{\mathbb{E}_{1,T}}&=& \|u\|_{W^1_p(J; W^{1-\frac1p}_p(\Sigma))\cap L_p(J,W^{4-\frac1p}_p(\Sigma))} + \|u(0)\|_{{X}_\gamma},\\
  \|u\|_{\widetilde{\mathbb{E}}_{1,T}}&=& \|u\|_{W^1_p(J; W^1_p(\Sigma_a))\cap L_p(J,W^4_p(\Sigma_a))} + \|u(0)\|_{\widetilde{X}_\gamma},
\end{eqnarray*}
then the operator norm of the embedding (\ref{eq:BUCEmbedding}) is bounded in $T>0$.
Additionally, we have
\begin{equation*}
  \widetilde{\mathbb{E}}_{1,T}\hookrightarrow C^{1-\frac1p}([0,T];W^1_p(\Sigma_a)).
\end{equation*}
Interpolation with \eqref{eq:BUCEmbedding} implies
\begin{equation*}
  \widetilde{\mathbb{E}}_{1,T}\hookrightarrow C^{\tau}([0,T];B^{2+\frac{n}p}_{p,1}(\Sigma_a))\hookrightarrow C^{\tau}([0,T];C^2(\ol{\Sigma_a}))
\end{equation*}
for some $\tau>0$ since $p>\frac{n+3}2$. Here again all operator norms of the embeddings are bounded in $T>0$. We will need the following technical lemma:
\begin{lemma}
  For every $\eps>0$ the extension operator $E$ above can be chosen such that for every $0<T<\infty$
  \begin{equation*}
    \sup_{0\leq t\leq T}\|h(t,\Pi(\cdot))-E{h}(t,\cdot)\|_{C^1(\Sigma_a)}\leq \eps\|h\|_{\mathbb{E}_{1,T}}.
  \end{equation*}
\end{lemma}
\begin{proof}
  First of all, since $E{h}(t,x)=h(t,x)$ for all $x\in\Sigma$, $t\in [0,T]$,
  \begin{equation*}
    \sup_{0\leq t\leq T}\|h(t,\Pi(\cdot))-E{h}(t,\cdot)\|_{C^1(\Sigma_{a'})}\leq a' \sup_{0\leq t\leq T}\|E{h}(t,\cdot)\|_{C^2(\Sigma_{a'})}
\leq Ca' \|h\|_{\mathbb{E}_{1,T}}
  \end{equation*}
  for any $0<a'\leq a$, where $C$ is independent of $0<T<\infty$.
  Hence, if, for given $\eps>0$, $a'$ is chosen sufficiently small, we have
  \begin{equation}\label{eq:EpsEstim}
    \sup_{0\leq t\leq T}\|h(t,\Pi(\cdot))-E{h}(t,\cdot)\|_{C^1(\Sigma_{a'})}\leq \eps \|h\|_{\mathbb{E}_{1,T}}.
  \end{equation}
  If we now define $E'\colon {\E}_{1,T}\to \widetilde{\E}_{1,T}$ by
  \begin{equation*}
    (E' h)(t,x)= (Eh)\left(t,\Pi(x)+ \frac{a'}{a}d_\Sigma(x)\nu_\Sigma(\Pi(x))\right)\quad \text{for all}\ x\in \Sigma_a, t\in [0,T],
  \end{equation*}
  then $E'\colon {\E}_{1,T}\to \widetilde{\E}_{1,T} $ is an extension operator, which satisfies the statement of the lemma.
\end{proof}

For technical reasons, we modify  the Hansawa transformation $\Theta_h$ to
\begin{equation*}
  \widetilde{\Theta}_h(x,t)= x+\chi(d_{\Sigma}(x)/a)\tilde{h}(t,x)\nu_\Sigma(\Pi(x)),
\end{equation*}
where $\tilde{h}=Eh\in \widetilde{E}_{1,T}$ is the extension of $h$ to $\Omega$ as above.
Then
\begin{eqnarray*}
  \|\widetilde{\Theta}_h(.,t)-\Theta_h(.,t)\|_{C^1(\ol\Omega)}&\leq& C\|\tilde{h}(.,t)-h(\Pi(.),t)\|_{C^1(\Sigma_a)}
\end{eqnarray*}
for all $0\leq t\leq T$, where $C$ is independent of $h$ and $0<T<\infty$. If we now choose $\eps>0$ in \eqref{eq:EpsEstim} sufficiently small,
$\widetilde{\Theta}_h(.,t)\colon\Omega\to\Omega$ is again a $C^1$-diffeomorphism for every $0\leq t\leq T$. This can be shown by applying the contraction mapping principle to
\begin{equation*}
  x =\Theta_h^{-1}\left(\Theta_h(x)-\widetilde{\Theta}_h(x)+y\right)
\end{equation*}
for given $y\in \Omega$, which is equivalent to $\widetilde{\Theta}_h(x)=y$.
 Moreover, $\widetilde{\Theta}_h(\Sigma,t)= \Theta_h(\Sigma,t)=\Gamma(t)$ for all $0\leq t\leq T$.

Now let
\begin{equation*}
F_{h,t}=\widetilde{\Theta}_{h}(.,t)\circ \widetilde{\Theta}_{h}(.,0)^{-1}.
\end{equation*}
 Then
$F_{h,t}\colon \Omega\to \Omega$ with $F_{h,t}(\Omega^\pm_0)= \Omega^\pm(t)$ and $F_{h,t}(\Gamma_0)=\Gamma(t)$, where $\Gamma_0=\Gamma(0)=\partial\Omega^+(0)$.
Moreover, $F_h=(F_{h,t})_{t\in [0,T]}\in BUC([0,T];W^{4-\frac3p}_p(\Omega))\cap W^1_p(0,T;W^1_p(\Omega))$
and
\begin{eqnarray}\label{eq:LipschitzA}
  \|F_{h_1}-F_{h_2}\|_{C^\tau([0,T];C^2(\ol{\Omega}))}&\leq& C\|h_1-h_2\|_{\mathbb{E}_{1,T}},\\\label{eq:Lipschitz2}
  \|F_{h_1}-F_{h_2}\|_{W^1_p(0,T;W^1_p(\Omega))}&\leq& C\|h_1-h_2\|_{\mathbb{E}_{1,T}}
\end{eqnarray}
for all $\|h_j\|_{\mathbb{E}_{1,T}}\leq R$, $j=1,2$, where $C$ is independent of $h_j$ and $0<T<\infty$. Since $F_{h,0}=\Id_{\Omega}$ for all $h\in \mathbb{E}_{1,T}$, \eqref{eq:LipschitzA} implies
\begin{equation}\label{eq:Lipschitz1}
  \|F_{h_1}-F_{h_2}\|_{BUC([0,T];C^2(\ol{\Omega}))}\leq CT^\tau\|h_1-h_2\|_{\mathbb{E}_{1,T}}.
\end{equation}

Now we consider
\begin{alignat*}{2}
  \partial_t v +  v\cdot \nabla v -\mu^\pm \Delta v+\nabla \tilde{p} &= 0&\quad& \text{in} \ \Omega^\pm(t),t\in (0,T),\\
  \Div v &= 0&\quad& \text{in} \ \Omega^\pm(t),t\in (0,T),\\
  \Jump{v} &= 0 &\quad& \text{on} \ \Gamma(t),t\in (0,T),\\
\Jump{\nu_{\Gamma(t)}\cdot T(v,\tilde{p})} &= \STC \MC_{\Gamma(t)}\nu_{\Gamma(t)}&\quad& \text{on} \ \Gamma(t),t\in (0,T),\\
v|_{\partial\Omega} &=0 && \text{on}\ \partial\Omega,t\in (0,T),\\
v|_{t=0} &= v_0  &\quad& \text{on} \ \Omega^\pm(t),t\in (0,T).
\end{alignat*}
 Defining
 \begin{equation*}
u(x,t)=v(F_{t,h}(x),t),\quad q(x,t)=\tilde{p}(F_{t,h}(x),t),
 \end{equation*}
 the latter system can be transformed to
\begin{alignat}{2}\label{eq:TransNS1}
  \partial_t u -\mu^\pm \Delta u+\nabla q &= a^\pm(h;D_x)(u,q)+ \partial_t F_{h}\cdot \nabla_h u- u\cdot \nabla_h u&\quad& \text{in} \ Q_T^\pm,\\\label{eq:TransNS2}
  \Div u &= \Tr((I-A(h))\nabla u)=:g(h)u &\quad& \text{in} \ Q^\pm_T,\\
  \Jump{u} &= 0 &\quad& \text{on} \ \Gamma_{0,T},\\
\Jump{\nu_{\Gamma_0} \cdot T(u,q)} &= t(h;D_x) (u,q) + \sigma \widetilde{H}_{h}&\quad& \text{on} \ \Gamma_{0,T},\\
u|_{\partial\Omega} &=0&&\text{on}\ \partial\Omega_T,
\\\label{eq:TransNS5}
u|_{t=0} &= v_0  &\quad& \text{on} \ \Omega_0^\pm,
\end{alignat}
where $Q^\pm_T=(0,T)\times\Omega_0^\pm$, $\Omega_0^-=\Omega\setminus(\Gamma_0\cup \Omega_0^+)$, $\Gamma_{0,T}=(0,T)\times \Gamma_0$, $\partial\Omega_T=(0,T)\times \partial\Omega$.
Here
\begin{eqnarray*}
  a^\pm(h;D_x)(u,q)&=& \mu^\pm\Div_h(\nabla_h u)-\mu^\pm \Div \nabla u +(\nabla-\nabla_h)q,\\
 \nabla_h  &=& A(h)\nabla,\ \Div_h u = \Tr (\nabla_h u),\ A(h)=DF_{t,h}^{-T},\ \nu_h= \frac{A(h)\nu_{\Gamma_0}}{|A(h)\nu_{\Gamma_0}|},\\
t(h,D_x)(u,q)&=& [\![(\nu_{\Gamma_0}-\nu_h)\cdot (2\mu^\pm D u-qI)+2\nu_h\cdot \sym(\nabla u-\nabla_h u)]\!],\\
\tilde{\MC}_h(x) &=& \MC_{\Gamma(t)}(F_{h,t}(x))\nu_{\Gamma(t)}(F_{h,t}(x))\qquad \text{for all}\ x\in \Gamma_0.
\end{eqnarray*}
In the following let
$Y_T= Y_T^1\times Y_T^2$, where
\begin{eqnarray*}
  Y_T^1&=& \left\{ u\in BUC([0,T];H^1(\Omega)^n)\cap H^1(0,T;L_{2,\sigma}(\Omega)):u|_{\Omega_0^\pm} \in L_2(0,T;H^2(\Omega_0^\pm)^n)\right\}\\
Y_T^2&=& \left\{q\in L_2(0,T;L_{2,(0)}(\Omega)): \nabla q|_{\Omega_0^\pm} \in L_2 ((0,T)\times\Omega_0^\pm)^n\right\}.
\end{eqnarray*}

The main result of this section is:
\begin{theorem}\label{thm:Existence}
  Let $R>0$, $h_0\in U$. Then there is some $T_0=T_0(R)>0$ such that for every $0<T\leq T_0$ and $h\in \E_{1,T}\cap \U$ with $h|_{t=0}=h_0$ and $v_0\in H^1_{0}(\Omega)^n\cap L_{2,\sigma}(\Omega)$, $n=2,3$, with $\max\{\|h\|_{\E_{1,T}},\|v_0\|_{H^1_{0}(\Omega)}\}\leq R$ there is a unique solution $(u,p)=:\mathcal{F}_T(h,v_0)\in Y_T$ of \eqref{eq:TransNS1}-\eqref{eq:TransNS5}.
  Moreover, for every $\eps>0$
  \begin{equation*}
    \mathcal{F}_T\in BC^1(A_{\eps,R}\times\overline{B_{H^1_{0}}(0,R)} ;Y_T),
  \end{equation*}
  where
    \begin{equation*}
    A_{\eps,R}=\left\{ h\in \overline{B_{\mathbb{E}_{1,T}}(0,R)}: h(0)=h_0, \sup_{0\leq t\leq T}\|h(t)\|_{L^\infty(\Sigma)}\leq a-\eps\right\}.
  \end{equation*}
\end{theorem}

We can formulate (\ref{eq:TransNS1})-(\ref{eq:TransNS5}) as an abstract fixed-point equation
\begin{equation}\label{eq:FixedPoint}
  Lw = G(w;h,v_0)\qquad \text{in}\ Z_T
\end{equation}
for $w\in Y_T$, where
\begin{eqnarray*}
  L(u,q)&=&
  \begin{pmatrix}
    \partial_t u -\mu^\pm \Delta u+\nabla q\\
    \Div u \\
    \Jump{\nu_{\Gamma_0}\cdot T^\pm (u,q)}\\
    u|_{t=0}
  \end{pmatrix}\\
  G(u,q;h,v_0)&=&
  \begin{pmatrix}
    a^\pm(h;D_x)u+ \partial_t F_{h}\cdot \nabla_h u- u\cdot \nabla_h u\\
    g(h)u-\tfrac{1}{|\Omega|}\int_\Omega g(h)u\sd x\\
    t(h;D_x) (u,q) + \sigma \widetilde{H}_{h}\\ v_0
  \end{pmatrix}
\end{eqnarray*}
for all $w=(u,q)\in Y_T$, where $Z_T= Z_T^1\times Z_T^2\times Z_T^3\times Z_T^4 $,
\begin{eqnarray*}
  Z_T^1&=& L_2((0,T)\times \Omega_0)^n,\qquad Z_T^4=H^1_{0}(\Omega)^n\cap L_{2,\sigma}(\Omega),\\
  Z_T^2&=& L_2(0,T;H^1_{(0)}(\Omega_0))\cap H^1(0,T;H^{-1}_{(0)}(\Omega_0)),\\
  Z_T^3&=& L_2(0,T;H^{\frac12}_2(\Gamma_0)^n)\cap H^{\frac14}(0,T;L_2(\Gamma_0)^n),
\end{eqnarray*}
and $Z_T^4= H^1_0(\Omega)^n\cap L_{2,\sigma}(\Omega)$.
Here  $H^1_{(0)}(\Omega_0)= H^1(\Omega_0)\cap L_{2,(0)}(\Omega)$ is normed by $\|\nabla \cdot\|_{L_2}$, $H^{-1}_{(0)}(\Omega)= (H^1_{(0)}(\Omega))'$.

First of all, let us note that \eqref{eq:FixedPoint} implies \eqref{eq:TransNS1}-\eqref{eq:TransNS5} except that \eqref{eq:TransNS2} is replaced by
\begin{equation*}
  \Div u = g(h)u -\frac1{|\Omega|} \int_\Omega g(h) \sd x.
\end{equation*}
But the latter equation implies \eqref{eq:TransNS2}, which can be seen as follows: Let $K(t)= \frac1{|\Omega|} \int_\Omega g(h(x,t)) \sd x$ and $v(x,t)= u(F_h^{-1}(x,t),t)$. Then $v(t)\in H^1_0(\Omega)$ for all $t\in (0,T)$ and therefore
\begin{eqnarray*}
  0 &=& \int_\Omega \Div v(x,t) \sd x= \int_{\Omega} \Tr (A(h(x,t))\nabla u(x,t)) \det DF_h(x,t) \sd x\\
 &=& K(t) \int_{\Omega} \det DF_h(x,t) \sd x
\end{eqnarray*}
for all $t\in [0,T]$. Since the last integral is positive, we obtain $K(t)=0$ for all $t\in (0,T)$.

\begin{lemma}\label{lem:C1Estim}
  Let $R>0$, $\eps>0$, and let $Y_T,Z_T$, $h_0$  be as above. Moreover, let
  \begin{equation*}
    A_{\eps,R} =\left\{h\in \mathbb{E}_{1,T}: \sup_{0\leq t\leq T}\|h(t)\|_{L^\infty(\Sigma)}\leq a-\eps, h(0)=h_0, \|h\|_{\mathbb{E}_{1,T}}\leq R \right\}.
  \end{equation*}
Then there is some $T_0>0$ such that for every $0<T\leq T_0$ the mapping $G$ defined above is well-defined and
$$
G\in C^1(\overline{B_{Y_T}(0,R')}\times A_{\eps,R}\times\overline{B_{H^1_0\cap L_{2,\sigma}}(0,R')}; Z_T).
$$
Moreover, there are some $C,\alpha>0$ such that
  \begin{equation*}
    \|G(w_1;h,v_0)-G(w_2;h,v_0)\|_{Z_T} \leq CT^\alpha\|w_1-w_2\|_{Y_T}
  \end{equation*}
for every $w_1,w_2\in \overline{B_{Y_T}(0,R')}$, $0<T\leq T_0$, $h\in A_{\eps,R} $, and $v_0\in \overline{B_{H^1_0\cap L_{2,\sigma}}(0,R)}$.
\end{lemma}
\begin{proof}
  First of all, because of \eqref{eq:LipschitzA}, for any $\eps>0$ there are some $C,T_0>0$ such that
  \begin{equation*}
    \|D_x F_{h}-\operatorname{id}\|_{BUC([0,T];C^1(\ol{\Omega}))} \leq \eps\|h\|_{\mathbb{E}_{1,T}}
  \end{equation*}
  for all $0<T\leq T_0$, $\|h\|_{\mathbb{E}_{1,T}}\leq R$.
  Hence $F_{t,h}\colon\Omega\to\Omega$ is a $C^2$-diffeomorphism and $D_x F_h$ is invertible with uniformly bounded inverse for these $h,T$. Since matrix inversion is smooth on the set of invertible matrices,
  \begin{equation*}
    A(h)= DF_{h}^{-T}\in  C^\tau([0,T];C^1(\overline{\Omega}))
  \end{equation*}
  for some $\tau>0$  if $\|h\|_{\mathbb{E}_{1,T}}\leq R$. Moreover, interpolation of \eqref{eq:LipschitzA} and \eqref{eq:Lipschitz2} yields
  \begin{equation*}
   DF_h\in  C^{\frac12-\frac1{2p}+\frac{\tau}2}([0,T];C^0(\overline{\Omega}))
  \end{equation*}
due to $W^1_p(0,T;X)\hookrightarrow C^{1-\frac1p}([0,T];X)$, where the operator norm of the latter embedding is bounded in $0<T<\infty$ if $W^1_p(0,T;X)$ is normed by
\begin{equation*}
 \|f\|_{W^1_p(0,T;X)}:=\|(f,f')\|_{L_p(0,T;X)}+\|f(0)\|_{X}.
\end{equation*}
 Here we have also used that $\|f\|_{C^1(\overline{\Omega})}\leq C \|f\|_{C^0(\overline{\Omega})}^{\frac12}\|f\|_{C^2(\overline{\Omega})}^{\frac12}$ and $W^1_p(\Omega)\hookrightarrow C^0(\ol{\Omega})$.
 Hence
  \begin{equation*}
    A(h)= DF_{h}^{-T}\in  C^{\frac12-\frac1{2p}+\frac{\tau}2}([0,T];C^0(\overline{\Omega})).
  \end{equation*}
  Furthermore,
  \begin{alignat}{1}\label{eq:ASmoothness}
 A&\in BC^1(\overline{B_{\E_{1,T}}(0,R)};X)\quad \text{with}\\
\label{eq:X}
    X&= C^\tau([0,T];C^1(\overline{\Omega}))\cap W^1_p(0,T;L_p(\Omega))\cap  C^{\frac12-\frac1{2p}+\frac{\tau}2}([0,T];C^0(\overline{\Omega}))
  \end{alignat}
  again since matrix inversion is smooth.

Using the above observations, one easily obtains
\begin{equation*}
  \|(\nabla_{h}-\nabla)f\|_{L_2(0,T;H^k(\Omega_0^\pm))}\leq CT^\tau \|h\|_{\mathbb{E}_{1,T}} \|f\|_{L_2(0,T;H^{k+1}(\Omega_0^\pm))}
\end{equation*}
for all $f\in L_2(0,T;H^k(\Omega_0^\pm))$, $k=0,1$, $\|h\|_{\mathbb{E}_{1,T}}\leq R$. From this estimate, one derives
\begin{eqnarray*}
  \|a^\pm(h;D_x)(u,q)\|_{L_2((0,T)\times \Omega_0^\pm)}&\leq& CT^\tau\|h\|_{\mathbb{E}_{1,T}}\|(u,q)\|_{Y_T},\\
  \|g(h)u\|_{L_2(0,T;H^1(\Omega_0^\pm))}&\leq& CT^\tau\|h\|_{\mathbb{E}_{1,T}}\|u\|_{L_2(0,T;H^2(\Omega_0^\pm))},\\
\|v\cdot DF_{h}\nabla u\|_{L_2((0,T)\times \Omega_0^\pm)}&\leq& CT^{\frac1{2}}\|v\|_{L_\infty(0,T;W^1_p(\Omega_0^\pm))}\|u\|_{L_\infty(0,T;W^1_p(\Omega_0^\pm))}\\
&\leq& CT^{\frac1{2}}\|v\|_{Y^1_T}\|u\|_{Y^1_T}\\
\|\partial_tF_{h}\cdot \nabla u\|_{L_2((0,T)\times \Omega_0^\pm)} &= & \|\left(\partial_t F_{h}-\partial_t F_{0}\right)\cdot \nabla u\|_{L_2((0,T)\times \Omega_0^\pm)}\\
&\leq&  CT^{\frac12-\frac1{p}}\|h\|_{\mathbb{E}_{1,T}}\|u\|_{L_\infty(0,T;H^1(\Omega))}\\
&\leq& CT^{\frac12-\frac1{p}}\|h\|_{\mathbb{E}_{1,T}}\|u\|_{Y^1_T},
\end{eqnarray*}
where we have used (\ref{eq:Lipschitz2}) for the last estimate. Moreover,
\begin{eqnarray}
  \|(t(h,D_x)u\|_{Z^2_T}&\leq& CT^\alpha\|h\|_{\mathbb{E}_{1,T}}\|(u,q)\|_{Y_T}
\end{eqnarray}
for some $\alpha>0$ can be proved in the same way as in \cite[Proof of Lemma 4.3]{FreeSurface}.
In order to estimate $g(h)u\in H^1(0,T;H^{-1}_{(0)}(\Omega))$, we use that
\begin{eqnarray*}
  (g(h)u,\varphi)_{\Omega} &=& - (u, \Div ((I-A(h)^T) \varphi))_{\Omega}\qquad \text{for all}\ \varphi\in H^1_{(0)}(\Omega).
\end{eqnarray*}
 Therefore we obtain for all $\varphi \in H^1_{(0)}(\Omega)$ with $\|\nabla\varphi\|_{L_2(\Omega)}=1$
\begin{eqnarray*}
  \frac{d}{dt}(g(h)u,\varphi)_{\Omega} &=&- (\partial_t u, \Div ((I-A(h)^T) \varphi))_{\Omega}- (\nabla u, (\partial_t A(h)^T)  \varphi)_{\Omega}\\
&\equiv& \weight{F_1(t),\varphi}+ \weight{F_2(t),\varphi},
\end{eqnarray*}
where
\begin{eqnarray*}
  \lefteqn{\left|(\partial_t u(t), \Div ((I-A(h(t))^T) \varphi))_{\Omega}\right|}\\
 &\leq& C\|\partial_t u(t)\|_{L_2(\Omega)}\|I-A(h)^T\|_{L_\infty(0,T;C^1(\ol\Omega))}\\
  &\leq& CT^{\tau}\|\partial_t u(t)\|_{L_2(\Omega)}\|h\|_{C^\tau([0,T;C^1(\ol\Omega))}
  \end{eqnarray*}
and
  \begin{eqnarray*}
\lefteqn{\left|(\nabla u (t), (\partial_t (I-A(h(t)))^T \varphi))_{\Omega}\right|}\\
&\leq & C\|u\|_{L_\infty(0,T;H^1(\Omega))}\|\|\partial_t E(h(t))\|_{W^1_p(\Omega)}\\
&\leq &C\|u\|_{Y_T^1}\|\partial_t A(h(t))\|_{W^1_p(\Omega)}
\end{eqnarray*}
for all $t\in (0,T)$.  Hence
\begin{eqnarray*}
  \|F_1\|_{L_2(0,T;H^{-1}_{(0)})}&\leq& C T^\tau\|\partial_t u\|_{L_2(\Omega\times (0,T))}\|h\|_{C^\tau([0,T];C^1(\ol\Omega))} \\
  \|F_2\|_{L_2(0,T;H^{-1}_{(0)})}&\leq& C\|u\|_{Y_T^1}\|\partial_t A(h(t))\|_{L_2(0,T;W^1_p)}\leq C T^{\frac12-\frac1p}\|u\|_{Y_T^1}\|h\|_{\E_{1,T}}.
\end{eqnarray*}
and therefore
\begin{equation*}
      \|\partial_t g(h)u\|_{L_2(0,T;H^{-1}_{(0)})} \leq C(R) T^{\min(\tau,\frac12-\frac1{p})}\|h\|_{\E_{1,T}}\|u\|_{Y^1_T}
\end{equation*}
for all $h\in \E_{1,T}$ with $\|h\|_{\E_{1,T}}\leq R$. Here we have used that $A\in BC^1(A_{\eps,R};X)$, where $X$ is as in \eqref{eq:X}.

Finally, it remains to estimate the term $\tilde{H}_h$. To this end we use that
\begin{equation*}
  \tilde{H}_h = \left(K(h)\circ \theta_{h_0}^{-1}\right)\nu_h
\end{equation*}
where ${\theta}_{h_0}:= \tilde{\Theta}_h(\cdot,0)|_{\Sigma}\colon \Sigma\to \Gamma_0$ bijectively. Here
\begin{equation*}
  K \in BC^1(A_{\eps,R};H^{\frac14}(0,T;L_2(\Sigma))\cap L_2(0,T;H^{\frac12}(\Sigma)))
\end{equation*}
because of Corollary~\ref{cor:Curvature}. Since ${\theta}_{h_0}\in C^2(\Sigma)^n$ is independent of $t$ and $h$, the same is true for $K(\cdot)\circ {\theta}_{h_0}^{-1}$ with $\Sigma$ replaced by $\Gamma_0$. Because of \eqref{eq:ASmoothness}, we have for $\tilde{H}(h):= \tilde{H}_h$ for all $h\in A_{\eps,R}$
\begin{equation*}
  \tilde{H} \in BC^1(\overline{B_{\mathbb{E}_{1,T}}(0,R)};H^{\frac14}(0,T;L_2(\Gamma_0))\cap L_2(0,T;H^{\frac12}(\Gamma_0))).
\end{equation*}
Altogether, since all terms in $G$ are linear or bilinear in $(u,q)$ and $A(h)$, these considerations imply that $G\in BC^1(\overline{B_{Y_T}(0,R)}\times A_{\eps,R}\times  \overline{B_{H^1(\Omega)^n}(0,R)}; Z_T)$ and
\begin{equation}\label{eq:GEstim}
  \|G(w_1;h,v_0)-G(w_2;h,v_0)\|_{Z_T}\leq CT^{\alpha'}\|w_1-w_2\|_{Y_T}
\end{equation}
for all $w_j=(u_j,q_j)\in Y_T$ with $\|w_j\|_{Y_T}\leq R$, $h\in A_{\eps,R}$, $v_0\in \overline{B_{Z_4}(0,R)}$ and $0<T\leq T_0$ for some $\alpha'>0$.
\end{proof}

\begin{proof*}{of Theorem~\ref{thm:Existence}}
Let $\eps>0$. Using Lemma~\ref{lem:C1Estim} and choosing $T_0>0$ sufficiently small,
\begin{equation*}
  L^{-1}G(.;h,v_0)\colon \overline{B_{Y_T}(0,R')}\to \overline{B_{Y_T}(0,R')}
\end{equation*}
becomes a contraction and is invertible if $h\in A_{\eps,R}$ and $\|v_0\|_{H^1(\Omega)}\leq R$, where
\begin{equation*}
  R'=2\sup\left\{\|L^{-1}G(0;h,v_0)\|_{Y_T}: h\in A_{\eps,R},\|v_0\|_{H^1(\Omega)}\leq R\right\}.
\end{equation*}
 Hence for every $(h,v_0)\in A_{\eps,R}\times\overline{B_{Z_4}(0,R)}$
there is a unique $w=:\mathcal{F}_T(h,v_0)\in \overline{B_{Y_T}(0,R')}$ such that
\begin{equation*}
  w= L^{-1}G(w;h,v_0).
\end{equation*}
Moreover, \eqref{eq:GEstim} implies
\begin{equation*}
  \|L^{-1}D_wG(w;h,v_0)\|_{\mathcal{L}(Z_T)}\leq CT^{\alpha'}\leq \frac12
\end{equation*}
for all $w_j=(u_j,q_j)\in Y_T$ with $\|w_j\|_{Y_T}\leq R$, $(h,v_0)\in A_{\eps,R}\times\overline{B_{H^1_{0}\cap L_{2,\sigma}}(0,R)}$, and $0<T\leq T_0$ if $T_0$ is sufficiently small. Hence we can apply the implicit function theorem to
\begin{equation*}
  F(w;h,v_0)= w- L^{-1}G(w;h,v_0)=0
\end{equation*}
and conclude that
\begin{equation*}
\mathcal{F}_T\in BC^1\left(A_{\eps,R}\times \overline{B_{Z_4}(0,R)}; \overline{B_{Y_T}(0,R')}\right)
\end{equation*}
 since $D_wF(w;h,v_0)$ is invertible for all $w\in \ol{B_{Y_T}(0,R')}$, $h\in A_{\eps,R}, v_0\in \overline{B_{Z_T^4}(0,R)}$.
\end{proof*}

Finally we obtain that the mapping $h\mapsto (\nu_{h}\cdot u)\circ  (\Theta_h|_{t=0})|_{\Sigma}$ satisfies the conditions to apply the general result of \cite{AmannQuasiLinear}:
\begin{corollary}\label{cor:Volterra}
    Let $R,\eps>0$, $T_0=T_0(R)>0$, $A_{\eps,R}$, and $\mathcal{F}_T$ be as in Theorem~\ref{thm:Existence}. For every $h\in A_{\eps,R}$, $v_0\in H^1_{0}(\Omega)^n\cap L_{2,\sigma}(\Omega)$ with $\|v_0\|_{H^1_{0}(\Omega)}\leq R$ let
    \begin{equation*}
      \mathcal{G}_T(h;v_0):= (\nu_{h}\cdot u)\circ  (\Theta_h|_{t=0})|_{\Sigma},
    \end{equation*}
 where $(u,p)=\mathcal{F}_T(h,v_0)$, $0<T\leq T_0$.
Then there is some $q>p$ such that $\mathcal{G}_T\in C^1(A_{\eps,R}\times\overline{B_{H^1_{0}\cap L_{2,\sigma}}(0,R)} ;L_q(0,T;X_0))$. Moreover, if $h_1|_{[0,T']}=h_2|_{[0,T']}$ for some $0<T'\leq T$, then $\mathcal{G}_T(h_1;v_0)|_{[0,T']}=\mathcal{G}_T(h_2;v_0)|_{[0,T']}$, i.e., the mapping $h\mapsto \mathcal{G}_T(h;v_0)$ is a Volterra map in the sense of \cite{AmannQuasiLinear}.
\end{corollary}
\begin{proof}
  First let $n=3$. Then by interpolation
  \begin{alignat*}{1}
    H^1(J;L_2(\Omega))&\cap L_2(J;H^2(\Omega))\hookrightarrow L_4(J;H^{\frac32}(\Omega))\cap L_2(J;H^2(\Omega))\\
&\hookrightarrow L_q(J;H^s(\Omega))\hookrightarrow L_q(J;W^1_p(\Omega))
  \end{alignat*}
where
\begin{equation*}
  s=1+\frac{n}2-\frac{n}p\in \left(\frac32,2\right),\quad \frac1q= \frac{n}2\left(\frac12-\frac1p\right)\in \left(0,\frac1p\right)
\end{equation*}
since $3<p<\frac{10}3$ and $n=3$. If $n=2$, we use that
\begin{equation*}
  H^1(J;L_2(\Omega))\cap L_2(J;H^2(\Omega))\hookrightarrow L_4(J;W^1_4(\Omega))\hookrightarrow L_4(J;W^1_p(\Omega)).
\end{equation*}
 Hence
$\mathcal{F}_T\in C^1(A_{\eps,R}\times\overline{B_{H^1_{0}}(0,R)} ;L_q(J;W^1_p(\Omega))$ for some $q>p$. The rest of the first statement follows from the trace theorem, the fact that
\begin{equation*}
h\mapsto \nu_h\in C^1(\overline{B_{\E_{1,T}}(0,R)},BUC([0,T];C^1(\ol{\Omega}))),
\end{equation*}
and that $\Theta_h|_{x\in\Sigma,t=0}\colon \Sigma\to \Gamma_0$ is a $C^2$-diffeomorphism.

Finally, the Volterra property follows easily from the fact that the solution of \eqref{eq:TransNS1}-\eqref{eq:TransNS5} on a time interval $(0,T)$ is also a solution of \eqref{eq:TransNS1}-\eqref{eq:TransNS5} on $(0,T')$ for any $0<T'<T$ (after restriction) and the uniqueness of the solution.
\end{proof}

\section{Local Well-Posedness}\label{sec:local}

In this section we show that the system \eqref{eq:1}-\eqref{eq:6} admits a unique local-in-time solution by reducing the whole system to a single quasilinear evolution equation for the height function $h$. For this purpose we use the solution operator obtained in the previous section and the solution operator for the (transformed) chemical potential coming from \eqref{eq:LWP3}-\eqref{eq:LWP4}.

We transform \eqref{eq:2'}, \eqref{eq:4}, \eqref{eq:4'} and \eqref{eq:5'} to the fixed domain $\Omega\backslash\Sigma$, with $\Sigma\subset\Omega$ as in the previous section, by means of the Hanzawa transform. This yields
\begin{alignat}{2} \label{eq:LWP1}
m\Delta_h\eta&=0 &\quad& \text{in}\ \Omega_T\backslash\Sigma_T,\\
\partial_t h-(\nu_h\cdot u)\circ  (\Theta_h|_{t=0})|_{\Sigma}&=-m\Jump{\nu_h\cdot\nabla_h\eta} &\quad& \text{on}\ \Sigma_{T},\\
\eta|_{\Sigma}&=\sigma K(h)&\quad& \text{on}\ \Sigma_{T},\\
\nu_\Omega\cdot m\nabla\eta&=0&\quad& \text{on}\ \partial\Omega_T,\\
h|_{t=0} &= h_0  &\quad& \text{on} \ \Sigma,\label{eq:LWP2}
\end{alignat}
where $\Sigma_T:=(0,T)\times\Sigma$, $\eta(t,x):=\mu(F_{t,h}(x),t)$ and $K(h)$ denotes the transformed mean curvature operator. Assume that we already know a solution $(u,h)\in Y_T^1\times\mathbb{E}_{1,T}$. Then we may use Corollary \ref{cor:Volterra} to write $(\nu_h\cdot u)\circ  (\Theta_h|_{t=0})|_{\Sigma}=\mathcal{G}_T(h;v_0)$. Consider the elliptic (time dependent) problem
\begin{alignat}{2}\label{eq:LWP3}
   m\Delta_h\eta&=0 &\quad& \text{in}\ \Omega_T\backslash\Sigma_T,\\
  \eta|_{\Sigma}&=\sigma K(h)&\quad& \text{on}\ \Sigma_{T},\\
\nu_\Omega\cdot m\nabla\eta&=0&\quad& \text{on}\ \partial\Omega_T.\label{eq:LWP4}
\end{alignat}
If
$$h\in \mathbb{E}_{1,T}^U:=\{h\in \mathbb{E}_{1,T}:h([0,T])\in U\},$$
where $U\subset W_p^{4-4/p}(\Sigma)$ is a sufficiently small neighborhood of zero, then \eqref{eq:LWP3}-\eqref{eq:LWP4} admits a unique solution $\eta=:S(h)K(h)\in L_p(0,T;W_p^{2}(\Omega\backslash\Sigma))$. Defining $B(h)\eta:=m\Jump{\nu_h\cdot\nabla_h\eta}$ we may reduce \eqref{eq:LWP1}-\eqref{eq:LWP2} to a single equation for $h$:
$$
\partial_t h+B(h)S(h)K(h)=\mathcal{G}_T(h;u_0)\ \text{on}\ \Sigma_{T},\quad h(0)=h_0\ \text{on}\ \Sigma.
$$
Employing the decomposition from Lemma \ref{lem:LCO1} we may write
\begin{equation}\label{eq:LWP5}
\partial_t h+\mathcal{A}(h)h=F_{1,T}(h)+F_2(h)\ \text{on}\ \Sigma_{T},\quad h(0)=h_0\ \text{on}\ \Sigma,
\end{equation}
where $\mathcal{A}(h):=B(h)S(h)P(h)$, $F_{1,T}(h):=\mathcal{G}_T(h;v_0)$ and $F_2(h):=-B(h)S(h)Q(h)$. Note that $\mathcal{A}$ and $F_2$ are nonlocal in $x$ but local in $t$, whereas $F_{1,T}$ is nonlocal operator in $t$ and $x$, but it has the Volterra property with respect to $t$. Firstly we show that $F_2\in C^1(U; W_p^{1-1/p}(\Sigma))$. By Lemma \ref{lem:LCO1} we have
$$Q\in C^1(U; W_p^{2-1/p}(\Sigma)).$$
Next we show that $S\in C^1(U;\mathcal{L}(W_p^{2-1/p}(\Sigma),W_p^2(\Omega\backslash\Sigma)))$. Writing
$$
\Delta_h=\sum_{j,k=1}^{n-1}a_{jk}^h\partial_j\partial_k+\sum_{j=1}^{n-1}a_j^h\partial_j
$$
with coefficients
$$a_{jk}^h=a_{jk}(x,h,\nabla h,\nabla^2 h),\quad a_{j}^h=a_{j}(x,h,\nabla h,\nabla^2 h),$$
depending smoothly on $(x,h,\nabla h,\nabla^2 h)$, it is not hard to see that
$$a_{jk}^h(\cdot,h,\nabla h,\nabla^2 h),a_{j}^h(\cdot,h,\nabla h,\nabla^2 h)\in BUC(\Omega\backslash\Sigma)^2,$$
for all $h\in U$. Here we used the embedding
$$W_p^{2-4/p}(\Sigma)\hookrightarrow C(\Sigma)$$
whenever $p>(n+3)/2$. This in turn yields that
$$h\mapsto\Delta_h\in C^1(U;\mathcal{L}(W_p^2(\Omega\backslash\Sigma),L_p(\Omega))).$$
We can now write
$$S(h)g=(\Delta_h,\gamma,\gamma_{N,\partial\Omega})^{-1}(0,g,0)$$
for some function $g\in W_p^{2-1/p}(\Sigma).$
Here $\gamma$ denotes the trace operator to $\Sigma$ and $\gamma_{N,\partial\Omega}$ stands for the Neumann derivative on $\partial\Omega$. Since the mapping $h\mapsto (\Delta_h,\gamma,\gamma_{N,\partial\Omega})$ belongs to
$$
C^1(U;\mathcal{L}(W_p^2(\Omega\backslash\Sigma),L_p(\Omega)\times W_p^{2-1/p}(\Sigma)\times W_p^{1-1/p}(\partial\Omega))),
$$
and inversion is smooth, we may conclude that
$$S\in C^1(U;\mathcal{L}(W_p^{2-1/p}(\Sigma),W_p^2(\Omega\backslash\Sigma))).$$
Finally, we show that $B\in C^1(U;\mathcal{L}(W_p^{2}(\Omega\backslash\Sigma),W_p^{1-1/p}(\Sigma)))$. We may write $B(h)=\sum_{j=1}^{n-1}b_j^h\gamma\partial_j$, where the coefficients $b_j^h=b_j(x,h,\nabla h)$ depend smoothly on $(x,h,\nabla h)$. This yields $b_j(\cdot,h,\nabla h)\in C^1(\Sigma),$
for each $h\in U$ since $W_p^{3-4/p}(\Sigma)\hookrightarrow C^1(\Sigma)$ for $p>(n+3)/2$.
It follows readily that
$$h\mapsto b_j^h\gamma\partial_j\in C^1(U;\mathcal{L}(W_p^{2}(\Omega\backslash\Sigma),W_p^{1-1/p}(\Sigma))).$$
Summarizing we have shown that
$$F_2\in C^1(U; W_p^{1-1/p}(\Sigma)),$$
hence the desired assertion.

Concerning the mapping $h\mapsto\mathcal{A}(h)$, we would like to show that
$$h\mapsto\mathcal{A}(h)\in C^1(U;\mathcal{L}(W_p^{4-1/p}(\Sigma),W_p^{1-1/p}(\Sigma))).$$
But this is an immediate consequence of Lemma \ref{lem:LCO1}, since
$$h\mapsto P(h)\in C^1(U;\mathcal{L}(X_1,W_p^{2-1/p}(\Sigma))).$$
It has been shown in \cite[Proof of Theorem 4.1]{PSZ09} that $\mathcal{A}(0)$ has the property of maximal $L_p$-regularity in $X_0=W_p^{1-1/p}(\Sigma)$, that is for each given $f\in L_p(0,T;X_0)$ there exists a unique solution $h\in H_p^1(0,T;X_0)\cap L_p(0,T;X_1)$ of the problem
$$\partial_t h(t)+\mathcal{A}(0) h(t)=f(t),\quad t\in (0,T),\quad h(0)=0,$$
where $X_1=W_p^{4-1/p}(\Sigma)$.
If $U\subset W_p^{4-4/p}(\Sigma)$ is a sufficiently small neighborhood of zero, then, by a perturbation argument, also $\mathcal{A}(h_0)$ has maximal $L_p$-regularity, whenever $h_0\in U$.

Note that the principal part in \eqref{eq:LWP5} is local in time. Furthermore, by Corollary \ref{cor:Volterra}, we have
$$F_{1,T}\in C^1(A_{\eps,R}\times\overline{B_{H^1_{0}\cap L_{2,\sigma}}(0,R)} ;L_q(0,T;X_0)),$$
for some $q>p$. This means that the nonlocal term $F_{1,T}$ is somehow of lower order with respect to $t$.
Based on this fact we are in a situation to apply existence and uniqueness results for quasilinear evolution equations with main part being local in time. We show that the nonlocal term $F_1$ satisfies the Lipschitz estimate
\begin{equation}\label{eq:LWP6}
\|F_{1,T}(h_1)-F_{1,T}(h_2)\|_{L_p(0,T;X_0)}\le\kappa(T)\|h_1-h_2\|_{\mathbb{E}_{1,T}}
\end{equation}
for all $h_1,h_2\in B_{r,\mathbb{E}_{1,T}}$, where $\kappa(T)\to 0_+$ as $T\to 0_+$ and
$$B_{r,\mathbb{E}_{1,T}}:=\{h\in \mathbb{E}_{1,T}: \|h-h_*\|_{\mathbb{E}_{1,T}}\le r,\ h(0)=h_0\},\ r\in  (0,1].$$
Here $h_*\in\mathbb{E}_{1,T_0}$ solves the linear Cauchy problem
$$\partial_t h_*(t)+\mathcal{A}(0)h_*(t)=0,\quad t\in (0,T_0),\quad h_*(0)=h_0,$$
for each $T_0>0$. Let $T\in (0,T_0)$, $\delta>0$ such that $\|h_0\|_{W_p^{4-4/p}(\Sigma)}<\delta$. It follows that
\begin{align*}
\|h(t)&\|_{W_p^{4-4/p}(\Sigma)}\\
&\le \|h(t)-h_*(t)\|_{W_p^{4-4/p}(\Sigma)}+\|h_*(t)-h_0\|_{W_p^{4-4/p}(\Sigma)}+
\|h_0\|_{W_p^{4-4/p}(\Sigma)}\\
&\le Mr+\sup_{t\in[0,T_0]}\|h_*(t)-h_0\|_{W_p^{4-4/p}(\Sigma)}+\delta\\
&\le a-\varepsilon,
\end{align*}
for all $h\in B_{r,\mathbb{E}_{1,T}}$, provided that $r,T_0,\delta>0$ are sufficiently small. Here $a>0$ denotes the number in the definition of the set $A_{\varepsilon,R}$ in Theorem \ref{thm:Existence}.

Choosing $R\ge r+\|h_*\|_{\mathbb{E}_{1,T_0}}$ we obtain that
$B_{r,\mathbb{E}_{1,T}}\subset A_{\varepsilon,R}$ for all $T\in (0,T_0)$. It holds that
$$F_{1,T}(h_1)-F_{1,T}(h_2)=\left[\int_0^1DF_{1,T}( h_2+\theta(h_1-h_2))d\theta\right](h_1-h_2).$$
Hence
\begin{align*}
\|F_{1,T}(h_1)&-F_{1,T}(h_2)\|_{L_q(0,T;X_0)}\le \|F_{1,T_0}(e(h_1))-F_{1,T_0}(e(h_2))\|_{L_q(0,T_0;X_0)}\\
& \le C\|e(h_1)-e(h_2)\|_{\mathbb{E}_{1,T_0}}\\
&\le CM\|h_1-h_2\|_{\mathbb{E}_{1,T}}+CM\|h_1(0)-h_2(0)\|_{W_p^{4-4/p}(\Sigma)}
\end{align*}
for all $h_1,h_2\in B_{r,\mathbb{E}_{1,T}}$, where
$$C:=\sup\{\|DF_{1,T_0}(h)\|_{\mathcal{L}(\mathbb{E}_{1,T_0};L_q(0,T_0;X_0))}:h\in B_{r,\mathbb{E}_{1,T_0}}\}>0.$$
and $e$ denotes an appropriate linear extension operator from $\mathbb{E}_{1,T}$ to $\mathbb{E}_{1,T_0}$, $T<T_0$, such that
$$\|e(h)\|_{\mathbb{E}_{1,T_0}}\le M\left(
\|h\|_{\mathbb{E}_{1,T}}+\|h(0)\|_{W_p^{4-4/p}(\Sigma)}\right)$$
holds for all $h\in \mathbb{E}_{1,T}$ and $M>0$ does not depend on $T<T_0$ and $h$ (see e.g. \cite[Lemma 7.2]{AmannQuasiLinear}).

Since $q>p$, an application of H\"{o}lder's inequality yields
\begin{align*}
\|F_{1,T}(h_1)-F_{1,T}(h_2)\|_{L_p(0,T;X_0)}&\le T^{\frac{q-p}{pq}}\|F_{1,T}(h_1)-F_{1,T}(h_2)\|_{L_q(0,T;X_0)}\\
&\le T^{\frac{q-p}{pq}}CM\|h_1-h_2\|_{\mathbb{E}_{1,T}}
\end{align*}
for all $h_1,h_2\in B_{r,\mathbb{E}_{1,T}}$.
Therefore we can choose $\kappa(T)=T^{\frac{q-p}{pq}}CM$. In particular, the nonlocal term $F_{1,T}(h)$ is a small perturbation in $L_p(0,T;X_0)$ provided that $T>0$ is small enough. This can be seen as follows
\begin{align*}\|F_{1,T}(h)\|_{L_p(0,T;X_0)}&\le \|F_{1,T}(h)-F_{1,T}(h_*)\|_{L_p(0,T;X_0)}+\|F_{1,T}(h_*)\|_{L_p(0,T;X_0)}\\
&\le\kappa(T)r+T^{\frac{q-p}{pq}}\|F_{1,T}(h_*)\|_{L_q(0,T;X_0)},
\end{align*}
for all $h\in B_{r,\mathbb{E}_{1,T}}$ and the right side of the last inequality can be made as small as we wish, by decreasing $T>0$.

We may now follow e.g.\ the lines of the proof of \cite[Theorem 2.1]{KPW10} to conclude that for each initial value $h_0\in U$ there exists a possibly small $T>0$ such that \eqref{eq:LWP5} admits a unique solution $h\in H_p^1(0,T;X_0)\cap L_p(0,T;X_1)$ which depends (locally) Lipschitz continuously on the initial data $h_0$.

We have proven the following result.
\begin{theorem}
Let $3<p\le 2(n+2)/n$, $n=2,3$, $R>0$ and $U=B_{W_p^{4-4/p}(\Sigma)}(0,\delta)$. Then there exist a sufficiently small $\delta>0$ and $T>0$ such that the (transformed) system \eqref{eq:TransNS1}-\eqref{eq:TransNS5}, \eqref{eq:LWP1}-\eqref{eq:LWP2} has a unique solution
$$(u,q,\eta,h)\in Y_T^1\times Y_T^2\times L_p(0,T;W_p^{2}(\Omega\backslash\Sigma))\times\mathbb{E}_{1,T},$$
provided that $h_0\in U$ and $v_0\in H^1_{0}(\Omega)^n\cap L_{2,\sigma}(\Omega)$, $\|v_0\|_{H^1}\le R$.
\end{theorem}

\section{Qualitative Behavior}\label{sec:Global}

This section is devoted to the long-time behavior of solutions to \eqref{eq:1}-\eqref{eq:6} starting close to equilibria. We will study the spectrum of the full linearization of the transformed two-phase Navier-Stokes/Mullins-Sekerka equations around an equilibrium. Since, among other things, the divergence-free-condition for the velocity field $v$ is destroyed under the Hanzawa transform, we have to split the solutions into two parts, one part which is divergence free and the remaining part which is not. The treatment of the first part is done by considering the so-called normal form of the equations in exponentially weighted spaces and the fact that the set of equilibria can be parameterized over the kernel of the linearization. The remaining part, which is not divergence free can be handled by the implicit function theorem.

For simplicity we assume that the dispersive phase is connected. Moreover, we assume for simplicity that $m=1$. (By a simple scaling in time one can always reduce to that case.) Note that the pressure $p$ as well as the chemical potential $\mu$ may be reconstructed by the semiflow $(v(t),\Gamma(t))$ as follows:
\begin{alignat*}{2}
\left(\nabla p|\nabla\phi\right)_{L_2} &= \left(\mu^\pm\Delta v-v\cdot\nabla v|\nabla \phi\right)_{L_2}&
\quad &\text{for all}\ \phi\in W^1_{2}(\Omega),\\
[\![ p]\!] &= 2[\![\mu^\pm (Dv)\nu_{\Gamma(t)}\cdot\nu_{\Gamma(t)}]\!]+\sigma H&
\quad &\text{ on } \Gamma(t),
\end{alignat*}
and
\begin{align*}
m\Delta \mu=0,&\quad t>0,\ x\in\Omega^\pm(t),\\
\mu|_{\Gamma(t)}=\sigma H,&\quad t>0,\ x\in\Gamma(t),\\
\nu_{\Omega}\cdot m\nabla \mu=0,&\quad t>0,\ x\in\partial\Omega.
\end{align*}
Therefore we may concentrate on the set of equilibria $\mathcal{E}$ for the flux $(v(t),\Gamma(t))$ which is given by
    $$\mathcal{E}=\{(0,S_R(x_0)),\ S_R(x_0)\subset\Omega\ \text{is a sphere}\}.$$
The linearization of the (transformed) two-phase Navier-Stokes-Mullins-Sekerka problem around an equilibrium $(0,\Sigma)\in\mathcal{E}$ reads as follows:
\begin{align}
\begin{split}\label{QB1}
\partial_t u-\mu^\pm\Delta u+\nabla q=f_u,&\quad t>0,\ x\in\Omega^\pm,\\
{\rm div}\ u=f_d,&\quad t>0,\ x\in\Omega^\pm,\\
-2[\![\mu^\pm Du]\!]\nu_\Sigma+[\![q]\!]\nu_\Sigma-\sigma(\mathcal{A}_\Sigma h)\nu_\Sigma=g_u,&\quad t>0,\ x\in\Sigma,\\
[\![u]\!]=0,&\quad t>0,\ x\in\Sigma,\\
u=0,&\quad t>0,\ x\in\partial\Omega,\\
\partial_t h-u\cdot\nu_\Sigma-[\![\partial_{\nu_{\Sigma}}\eta]\!]=g_h,&\quad t>0,\ x\in\Sigma,\\
\Delta \eta=f_\eta,&\quad t>0,\ x\in\Omega^\pm,\\
\eta|_{\Sigma}+\mathcal{A}_{\Sigma}h=g_\eta,&\quad t>0,\ x\in\Sigma,\\
\partial_\nu\eta=0,&\quad t>0,\ x\in\partial\Omega,\\
u(0)=u_0,&\quad x\in\Omega^\pm,\\
h(0)=h_0,&\quad x\in\Sigma,
\end{split}
\end{align}
where $\mathcal{A}_\Sigma=\frac{n-1}{R^2}I+\Delta_\Sigma$ and $\Delta_\Sigma$ denotes the Laplace-Beltrami operator on $\Sigma$. We want to reformulate \eqref{QB1} as an abstract evolution equation. To this end we introduce the Banach spaces $X_0:=L_{2,\sigma}(\Omega)\times W_p^{1-1/p}(\Sigma)$ and $\tilde{X}_1:=(L_{2,\sigma}(\Omega)\cap W_2^2(\Omega\setminus\Sigma)^n)\times W_p^{4-1/p}(\Sigma)$, where
    $$L_{2,\sigma}(\Omega):=\overline{\{u\in C_0^\infty(\Omega)^n:{\rm div}\ u=0\}}^{\|\cdot\|_{L_2(\Omega)}}.$$
Define a linear operator $A:D(A)\subset\tilde{X}_1\to X_0$ by means of
    $$A(u,h):=(-\mu^\pm\Delta u+\nabla q,-u\cdot\nu_\Sigma-[\![\partial_{\nu_{\Sigma}}\eta]\!]),$$
with domain
    $$D(A)=\{(u,h)\in \tilde{X}_1:u=0\ \text{on}\ \partial\Omega,\ [\![u]\!]=0\ \text{on}\ \Sigma\}.$$
Here $q\in H_{(0)}^1(\Omega\backslash\Sigma)$ and $\eta\in W_p^2(\Omega\backslash\Sigma)$ are determined as the solutions of the elliptic transmission problems
\begin{alignat*}{2}
\left(\nabla q|\nabla\phi\right)_{L_2} &= \left(\mu^\pm\Delta u|\nabla \phi\right)_{L_2}
&\quad&\text{for all}\ \phi\in H^1(\Omega),\\
[\![ q]\!] &= 2[\![\mu^\pm (Du)\nu_\Sigma\cdot\nu_\Sigma]\!]
+\sigma\cA_\Sigma h&\quad&\mbox{ on } \Sigma,
\end{alignat*}
and
\begin{align*}
\Delta \eta=0,&\quad t>0,\ x\in\Omega\setminus\Sigma,\\
\eta|_{\Sigma}+\mathcal{A}_{\Sigma}h=0,&\quad t>0,\ x\in\Sigma,\\
\partial_\nu\eta=0,&\quad t>0,\ x\in\partial\Omega.\\
\end{align*}
In the sequel we will use the solution formula
$$\nabla q=\mathcal{T}_1 (\mu^\pm\Delta u)+\mathcal{T}_2(2[\![\mu^\pm (Du)\nu_\Sigma\cdot \nu_\Sigma]\!]).$$
Setting $z=(u,h)$ and $f=(f_u,g_h)$ we may rewrite \eqref{QB1} as
    \begin{equation}\label{QB2}
    \dot{z}(t)+Az(t)=f(t),\ t>0,\quad z(0)=z_0:=(u_0,h_0),
    \end{equation}
provided that $f_d=g_u=f_\eta=g_\eta=0$. The operator $A$ has the following properties.
\begin{proposition}\label{propertiesA}
Let $n=2,3$, $p\in (3,2(n+2)/n)$, $\mu^\pm >0, \sigma>0$ be constants
and let $X_0$ and $A$ be defined as above.
Then the following assertions hold.
\begin{enumerate}
\item The linear operator $-A$ generates an analytic
$C_0$-semigroup $e^{-At}$ in $X_0$ which has the property of maximal
$L_p$-regularity.
\item The spectrum of $A$ consists of countably many eigenvalues
with finite algebraic multiplicity and is independent of $p$.
\item $-A$ has no eigenvalues $\lambda$
with nonnegative real part other than $\lambda=0$.
\item $\lambda=0$ is a semi-simple eigenvalue with multiplicity $n+1$, i.e.\ $X_0=N(A)\oplus R(A)$.
\item The kernel $N(A)$ is isomorphic to the tangent space
$T_{z_*}\mathcal{E}$ of $\cE$ at the given equilibrium $z_*=(0, \Sigma)\in\mathcal{E}$.
\item The restriction of $e^{-At}$ to $R(A)$ is exponentially stable.
\end{enumerate}
\end{proposition}
\begin{proof}
Consider \eqref{QB1} with $f_d=g_u=f_\eta=g_\eta=0$ and let $J=(0,T)$, $T>0$. Suppose that
$$h\in W_p^1(J;W_p^{1-1/p}(\Sigma))\cap L_p(J;W_p^{4-1/p}(\Sigma))$$
is known. Then solve problem $\eqref{QB1}_{1}$-$\eqref{QB1}_{5}$ with initial value $u_0\in H_{0}^1(\Omega)^n\cap L_{2,\sigma}(\Omega)$ by Theorem \ref{thm:MaxReg} with $g=a=0$ to obtain a unique solution
$$
u=S_{1,T}(h)\in H^1(0,T;L_{2,\sigma}(\Omega))\cap L_\infty(0,T;H^1_{0}(\Omega)^n)\cap L_2(0,T;H^2(\Omega\setminus\Sigma)^n),
$$
for each $T>0$. Plugging $u=S_{1,T}(h)$ into $\eqref{QB1}_6$ and denoting by $\eta=S_2(-\mathcal{A}_\Sigma h)=-S_2(\mathcal{A}_\Sigma h)$ the unique solution to $\eqref{QB1}_{7,8,9}$, we obtain the linear nonlocal problem
\begin{equation}\label{QB2a}
\partial_t h-m\Jump{\partial_{\nu_\Sigma}S_2(\mathcal{A}_\Sigma h)}=S_{1,T}(h)+g_h\ \text{on}\ \Sigma_{T},\quad h(0)=h_0\ \text{on}\ \Sigma.
\end{equation}
By \cite[Proof of Theorem 4.1]{PSZ09} the operator $[h\mapsto m\Jump{\partial_{\nu_\Sigma}S_2(\mathcal{A}_\Sigma h)}]$ has maximal $L_p$-regularity. Furthermore it holds that $S_{1,T}(h)\in L_q(0,T;W_p^{1-1/p}(\Sigma))$ for some $q>p$ which means that this term is of lower order in $L_p(0,T;X_0)$ compared to $m\Jump{\partial_{\nu_\Sigma}S_2(\mathcal{A}_\Sigma h)}]$. This can be seen as in the proof of Corollary \ref{cor:Volterra}.

Hence, by perturbation arguments we may conclude that \eqref{QB2a} has for each given $h_0\in W_p^{4-4/p}(\Sigma)$ a unique solution
$$h\in W_p^1(J;W_p^{1-1/p}(\Sigma))\cap L_p(J;W_p^{4-1/p}(\Sigma)).$$
In other words, we have shown that for each $T>0$ and for each given $f=(f_u,g_h)\in L_2(J;L_{2,\sigma}(\Omega))\times L_p(J;W_p^{1-1/p}(\Sigma))$ there exists a unique solution $z=(u,h)$ of \eqref{QB2} with
$$u\in H^1(0,T;L_{2,\sigma}(\Omega))\cap L_\infty(0,T;H^1_{0}(\Omega)^n)\cap L_2(0,T;H^2(\Omega\setminus \Sigma )^n),$$
and
$$h\in W_p^1(J;W_p^{1-1/p}(\Sigma))\cap L_p(J;W_p^{4-1/p}(\Sigma)),$$
provided that $u_0\in L_{2,\sigma}(\Omega)\cap H_0^1(\Omega)^n$ and $h_0\in W_p^{4-4/p}(\Sigma)$.

Mimicking the proof of \cite[Proposition 1.2]{Pr02} it follows that the operator $-A$ generates an analytic semigroup in $X_0=L_{2,\sigma}(\Omega)\times W_p^{1-1/p}(\Sigma)$.

By compact embedding, the resolvent of $A$ is compact and therefore the spectrum $\sigma(A)$ of $A$ consists of countably many eigenvalues with finite algebraic multiplicity and $\sigma(A)$ does not depend on $p$, by classical results. Let $\lambda\in \sigma(-A)$ with eigenfunctions $(u,h)$. Then the corresponding eigenvalue problem is
\begin{align}
\begin{split}\label{QB3}
\lambda u-\Delta u+\nabla q=0,&\quad x\in\Omega\setminus\Sigma,\\
{\rm div}\ u=0,&\quad x\in\Omega\setminus\Sigma,\\
-2[\![\mu^\pm(Du)]\!]\nu_\Sigma+[\![q]\!]\nu_\Sigma-\sigma(\mathcal{A}_\Sigma h)\nu_\Sigma=0,&\quad x\in\Sigma,\\
[\![u]\!]=0,&\quad x\in\Sigma,\\
\lambda h-u\cdot\nu_\Sigma-[\![\partial_{\nu_{\Sigma}}\eta]\!]=0,&\quad x\in\Sigma,\\
\Delta \eta=0,&\quad x\in\Omega\setminus\Sigma,\\
\eta|_{\Sigma}+\mathcal{A}_{\Sigma}h=0,&\quad x\in\Sigma,\\
\partial_\nu\eta=0,&\quad x\in\partial\Omega,\\
u=0,&\quad x\in\partial\Omega.
\end{split}
\end{align}
Taking the inner product of equation $\eqref{QB3}_1$ with $u$, integrating by parts and invoking the boundary as well as the transmission conditions, we obtain
    \begin{equation}\label{QB4}
    \lambda \|u\|_{2}^2+2\|\mu^\pm Du\|_2^2+\sigma\|\nabla\eta\|_2^2-\sigma\bar{\lambda}(\mathcal{A}_\Sigma h|h)_{L_2(\Sigma)}=0.
    \end{equation}
If $\lambda\neq 0$, then
    $$\lambda\int_\Sigma h do=\int_\Sigma \nu_{\Sigma}\cdot u do+\int_\Sigma [\![\partial_{\nu_{\Sigma}}\eta]\!] do=\int_{\Omega^+}{\rm div}\, u\, dx=0,$$
hence $h$ has mean value zero. It is well-known that the operator $\mathcal{A}_{\Sigma}=\frac{n-1}{R^2}+\Delta_\Sigma$ is negative definite on $L_{2,(0)}(\Omega)$. Taking real parts in \eqref{QB4} it follows that $\eta=const$ and $Du=0$, hence $u=0$ by Korn's inequality since $u|_{\partial\Omega}=0$. This in turn yields $h=0$ by $\eqref{QB3}_5$, showing that there are no eigenvalues $\lambda\neq 0$ of $-A$ with $\rm{Re}\ \lambda\ge 0$. Next we show that $\lambda=0$ is an eigenvalue of $A$. If $\lambda=0$, then \eqref{QB4} implies $\eta=\eta_\infty=const$ and $Du=0$. Hence, as before, $u=0$ by Korn's inequality. Since $q$ is constant by $\eqref{QB3}_1$ it follows from $\eqref{QB3}_{3,7}$ that
    $$\eta_\infty=\frac{[\![q]\!]}{\sigma}=\frac{n-1}{R^2}h+\Delta_\Sigma h,$$
which is a linear second order partial differential equation for $h$ on $\Sigma$. Note that a special solution to this linear equation is given by the constant function $h_\infty=\eta_\infty R^2/(n-1)$. The solution space $\mathcal{L}$ of the corresponding homogeneous equation $\mathcal{A}_\Sigma h=0$ is given by
    $$\mathcal{L}={\rm span}\{Y_1,\ldots,Y_n\},$$
where $Y_j$, $j\in\{1,\ldots,n\}$, are the spherical harmonics of degree one. Furthermore it holds that $\dim\mathcal{L}=n$. Since the constant $\eta_\infty=[\![q]\!]/\sigma$ is arbitrary, we see that $\dim N(A)=n+1$.

Let $z_1\in N(A)$ such that $A z=z_1$. The corresponding problem for $z=(u, h)^{\sf T}$ is given by
\begin{align}
\begin{split}\label{QB5}
-\Delta u+\nabla q=0,&\quad x\in\Omega\setminus\Sigma,\\
{\rm div}\ u=0,&\quad x\in\Omega\setminus\Sigma,\\
-2[\![\mu Du]\!]\nu_\Sigma+[\![q]\!]\nu_\Sigma-\sigma(\mathcal{A}_\Sigma h)\nu_\Sigma=0,&\quad x\in\Sigma,\\
[\![u]\!]=0,&\quad x\in\Sigma,\\
-u\cdot\nu_\Sigma-[\![\partial_{\nu_{\Sigma}}\eta]\!]=h_1,&\quad x\in\Sigma,\\
\Delta \eta=0,&\quad x\in\Omega\setminus\Sigma,\\
\eta|_{\Sigma}+\mathcal{A}_{\Sigma}h=0,&\quad x\in\Sigma,\\
\partial_\nu\eta=0,&\quad x\in\partial\Omega,\\
u=0,&\quad x\in\partial\Omega,
\end{split}
\end{align}
since $z_1=(0,h_1)$ and $h_1=\sum_{j=0}^n \alpha_j Y_j$, $Y_0:=1$. From the divergence condition we obtain
    $$0=\int_\Omega \Div u \sd x=\int_\Sigma\left(u\cdot\nu_\Sigma+[\![\partial_{\nu_{\Sigma}}\eta]\!]\right)do=-\int_\Sigma h_1 do,$$
hence $h_1$ has mean value zero and this in turn implies $\mathcal{A}_\Sigma h_1=0$. Multiplying $\eqref{QB5}_1$ by $u$, integrating by parts and taking into account the boundary and transmission conditions, we obtain
    \begin{equation}\label{QB6}
    2\|\mu^\pm Du\|_2^2+\sigma\|\nabla\eta\|_2^2+\sigma(\mathcal{A}_\Sigma h|h_1)_{L_2(\Sigma)}=0.
    \end{equation}
Since $\mathcal{A}_{\Sigma}$ is self-adjoint in $L_2(\Sigma)$ it follows that the last term in \eqref{QB6} vanishes and then, as before, $\eta=const$ and $u=0$, by Korn's inequality. In this case $\eqref{QB5}_5$ yields $h_1=0$, i.e.\ $z\in N(A)$, hence $N(A^2)=N(A)$. Since $A$ has compact resolvent, it follows that $R(A)$ is closed in $X_0$ and $\lambda=0$ is a pole of $(\lambda-A)^{-1}$. Therefore \cite[Remark A.2.4]{Lunardi} yields that $\lambda=0$ is semisimple, in particular it holds that $X_0=N(A)\oplus R(A)$. Moreover, the restricted semigroup $e^{-At}|_{R(A)}$ is exponentially stable, since we have a spectral gap.

Finally we show that the tangent space $T_{z_*}\mathcal{E}$ of $\mathcal{E}$ at $z_*=(0,\Sigma)\in\mathcal{E}$ coincides with $N(A)$. This can be seen as follows. Assume w.l.o.g.\ that $\Sigma$ is centered at the origin of $\mathbb{R}^n$ with radius $R$. Suppose $\mathcal{S}$ is a sphere that is sufficiently close to $\Sigma$. Denote by $(y_1,\ldots,y_n)$ the center of $\mathcal{S}$ and let $R+y_0$ be the corresponding radius of $\mathcal{S}$. Then by \cite[Section 6]{ES98} the sphere $\mathcal{S}$ can be parametrized over $\Sigma$ by the distance function
    $$d(y)=\sum_{j=1}^n y_jY_j-R+\sqrt{\left(\sum_{j=1}^n y_jY_j\right)^2+(R+y_0)^2-\sum_{j=1}^ny_j^2}.$$
Denoting by $\mathcal{O}$ a sufficiently small neighborhood of $0$ in $\mathbb{R}^{n+1}$, the mapping $d:\mathcal{O}\to W_p^{4-1/p}(\Sigma)$ is smooth and the derivative at $0$ is given by
    \begin{equation}\label{QB7}
    d'(0)w=\sum_{j=0}^n w_j Y_j,\quad \text{for all}\ w\in\mathbb{R}^{n+1}.
    \end{equation}
Therefore, near $\Sigma$, the set of equilibria $\mathcal{E}$ is a smooth manifold in $\tilde{X}_1$ of dimension $n+1$ and $T_{z_*}\mathcal{E}=N(A)$ by \eqref{QB7}.

Since $X_0=N(A)\oplus R(A)$ and $\sigma(A|_{R(A)})\subset \mathbb{C}_+$ it follows that the restricted semigroup $e^{-At}|_{R(A)}$ is exponentially stable. The proof is complete.

\end{proof}

We are now ready to prove the main result of this section. Note that the transformed equations near an equilibrium $(0,\Sigma)\in\cE$ read as follows.
	\begin{align}\label{nonlinstab1}
	\begin{split}
	\partial_t u-\mu^\pm\Delta u+\nabla \pi=F_u(u,\pi,h),&\quad t>0,\ x\in\Omega^\pm,\\
	{\rm div}\ u=F_d(u,h),&\quad t>0,\ x\in\Omega^\pm,\\
	-[\![\mu(\nabla u+\nabla u^{\sf T})]\!]\nu_\Sigma+[\![\pi]\!]\nu_\Sigma-\sigma(\mathcal{A}_\Sigma h)\nu_\Sigma=G_u(u,h),&\quad t>0,\ x\in\Sigma,\\
	[\![u]\!]=0,&\quad t>0,\ x\in\Sigma,\\
	\partial_t h-u\cdot\nu_\Sigma-[\![\partial_{\nu_{\Sigma}}\eta]\!]=G_h(u,\eta,h),&\quad t>0,\ x	\in\Sigma,\\
	\Delta \eta=F_\eta(\eta,h),&\quad t>0,\ x\in\Omega^\pm,\\
	\eta|_{\Sigma}+\mathcal{A}_{\Sigma}h=G_\eta(h),&\quad t>0,\ x\in\Sigma,\\
	\partial_\nu\eta=0,&\quad t>0,\ x\in\partial\Omega,\\
	u=0,&\quad t>0,\ x\in\partial\Omega,\\
	u(0)=u_0,&\quad x\in\Omega^\pm,\\
	h(0)=h_0,&\quad x\in\Sigma,
	\end{split}
	\end{align}
where the derivatives of the nonlinearities on the right hand side with respect to $(u,h)$ vanish at $(u,h)=(0,0)$ for constant $\pi$ and constant $\eta$.
\begin{theorem}\label{stability} The equilibrium $(0,\Sigma)\in\cE$ is stable in the sense
that for each $\varepsilon>0$
there exists some $\delta(\varepsilon)>0$ such that for all initial values
$(u_0,h_0)$ subject to
$$\|h_0\|_{W_p^{4-4/p}(\Sigma)}+\|u_0\|_{H^1_0(\Omega)^n}\leq \delta(\varepsilon)$$ there exists a unique global solution $(u(t),h(t))$
of \eqref{nonlinstab1} and it satisfies
$$\|h(t)\|_{W_p^{4-4/p}(\Sigma)}+\|u(t)\|_{H^1_0(\Omega)^n}\leq \varepsilon\ \text{for all}\ t\geq0.$$
Moreover, there exists some $h_\infty\in W_p^{4-1/p}(\Sigma)$ such that $\Theta_{h_\infty}\Sigma=\partial B_R(x)\subset\Omega$ for some $R>0$, $x\in\Omega$, and
$$\lim_{t\to\infty} \left(\|h(t)-h_\infty\|_{W_p^{4-4/p}(\Sigma)}+
\|u(t)\|_{H^1_0(\Omega)^n}\right)=0.$$
The convergence is at an exponential rate.
\end{theorem}
\begin{proof}
The nonlinear phase manifold for the semiflow is given by
	\begin{equation*}
          \mathcal{P}\mathcal{M}=\{(u,h)\in H^1_0(\Omega)^n\times W_p^{4-4/p}(\Sigma): \diver u=F_d(u,h)\}.
	\end{equation*}
In a first step we want to parametrize $\mathcal{P}\mathcal{M}$ over its tangent space at $(0,0)$, that is
	\begin{equation*}
          \mathcal{P}\mathcal{M}_0:=\{(u,h)\in H^1_0(\Omega)^n\times W_p^{4-4/p}(\Sigma): \diver u=0\}.
	\end{equation*}
To this end we consider the generalized Stokes equation
	\begin{align}\label{nonlinstab2}
	\begin{split}
	-\Delta u+\nabla\pi =0&\quad \text{in}\ \Omega,\\
	\diver u=f&\quad \text{in}\ \Omega,\\
	u=0&\quad \text{on}\ \partial\Omega
	\end{split}
	\end{align}
for which we have the following existence and uniqueness result.
\begin{proposition}
For every $f\in L^2_{(0)}(\Omega)=\{u\in L^2(\Omega): \int_\Omega u\sd x=0\}$ the Stokes problem \eqref{nonlinstab2} admits a unique solution $(u,\pi)\in H^1_0(\Omega)^n\times L^2_{(0)}(\Omega)$, which depends continuously on $f\in L^2_{(0)}(\Omega)$.
\end{proposition}
\begin{proof}
  The proposition is a special case of \cite[Theorem III.1.4.1]{BuchSohr}.
\end{proof}
With the help of this result we may continue as follows. For a given $(\tilde{u},\tilde{h})\in \mathcal{P}\mathcal{M}_0$ with a sufficiently small norm, we solve the auxiliary problem
\begin{align}\label{nonlinstab3}
\begin{split}
-\Delta \bar{u}+\nabla\bar{\pi} =0&\quad \text{in}\ \Omega,\\
\diver {\bar{u}}=P_0F_d(\bar{u}+\tilde{u},\tilde{h})&\quad \text{in}\ \Omega,\\
\bar{u}=0&\quad \text{on}\ \partial\Omega,
\end{split}
\end{align}
where $P_0:L^2(\Omega)\to L_{(0)}^2(\Omega)$ is defined by $P_0f=f-\frac{1}{|\Omega|}\int_\Omega f dx$.
Since the Fr\'{e}chet derivatives of the nonlinearities vanish in $(0,0)$, the implicit function theorem yields the existence of a ball $B(0,r)\subset H_0^1(\Omega)^n\cap W_p^{4-4/p}(\Sigma)$ and a unique solution
$$(\bar{u},\bar{\pi})=\tilde{\phi}(\tilde{u},\tilde{h})\in H^1_{0}(\Omega)\times L^2_{(0)}(\Omega)$$
with a function $\tilde{\phi}\in C^1(B(0,r))$ such that $\tilde{\phi}'(0)=0$.
Define $\Theta_{\tilde{h}}(x)$ as in \eqref{eq:Hansawa} with $h$ replaced by $\tilde{h}$, which does not depend on $t$. Let $v(x):=(\bar{u}+\tilde{u})(\Theta_{\tilde{h}}^{-1}(x))$. Then $v\in H_0^1(\Omega)^n$ and
\begin{align*}
\Div v(x)&=\Tr[D\Theta_{\tilde{h}}^{-T}(x)\nabla (\tilde{u}+\bar{u})(\Theta_{\tilde{h}}^{-1}(x))]\\
&=\Tr[(D\Theta_{\tilde{h}}^{-T}(x)-I)\nabla (\tilde{u}+\bar{u})(\Theta_{\tilde{h}}^{-1}(x))]+\Div (\bar{u}+\tilde{u})(\Theta_{\tilde{h}}^{-1}(x))\\
&=-F_d(\bar{u}+\tilde{u},\tilde{h})(\Theta_{\tilde{h}}^{-1}(x))+
P_0F_d(\bar{u}+\tilde{u},\tilde{h})(\Theta_{\tilde{h}}^{-1}(x))\\
&=-\frac{1}{|\Omega|}\int_\Omega F_d(\bar{u}+\tilde{u},\tilde{h})(x)\ dx,
\end{align*}
since $\Div\tilde{u}=0$. Because of $0=\int_\Omega\Div v(x) dx$, it follows that $P_0 F_d(\bar{u}+\tilde{u},\tilde{h})=F_d(\bar{u}+\tilde{u},\tilde{h})$.

Let
$$P\colon H^1_{0}(\Omega)^n\times L^2_{(0)}(\Omega)\to H^1_{0}(\Omega)^n,\ P(u,\pi)=u,$$
and set $\phi(\tilde{u},\tilde{h})=P\tilde{\phi}(\tilde{u},\tilde{h})$.
It is not difficult to see that $$(u,h):=(\tilde{u},\tilde{h})+(\phi(\tilde{u},\tilde{h}),0)\in \mathcal{P}\mathcal{M}.$$ Note that this mapping is injective. For the final construction of the parametrization we have to show that this mapping is also surjective. For that purpose we solve the linear problem
\begin{align}\label{nonlinstab4}
\begin{split}
-\Delta \bar{u}+\nabla\bar{\pi} =0&\quad \text{in}\ \Omega,\\
\diver {\bar{u}}=P_0F_d(u,h)&\quad \text{in}\ \Omega,\\
\bar{u}=0&\quad \text{on}\ \partial\Omega,
\end{split}
\end{align}
for given functions $(u,h)\in\mathcal{P}\mathcal{M}$. Setting $(\tilde{u},\tilde{h})=(u-\bar{u},h)$ we obtain that $\tilde{u}\in H_0^1(\Omega)^n$ and
$$\Div \tilde{u}=F_d(u,h)-P_0F_d(u,h)=\frac{1}{|\Omega|}\int_\Omega F_d(u,h) dx.$$
Since $0=\int_\Omega\Div \tilde{u}dx$ this yields $P_0 F_d(u,h)=F_d(u,h)$.

Furthermore it holds that $(\tilde{u},\tilde{h})\in\mathcal{P}\mathcal{M}_0$ and $\bar{u}=\phi(\tilde{u},\tilde{h})$ by injectivity. This in turn proves surjectivity. Observe that also $\phi(0)=0$. This can be seen as follows. Suppose that $\tilde{u}=\tilde{h}=0$. Then obviously $\bar{u}=0$ and $\bar{\pi}=const.$ is a solution of \eqref{nonlinstab3}. By the uniqueness it follows that $\phi(0)=0$. Furthermore, if $(u_\infty,h_\infty,\pi_\infty,\eta_\infty)$ is an equilibrium of \eqref{nonlinstab1}, then $u_\infty=0$ and
$$\eta_\infty=[\![\pi_\infty]\!]/\sigma=\mathcal{H}(h_\infty)=const.$$
Since $F_d(0,h_\infty)=0$, the unique solvability of \eqref{nonlinstab3} implies that $\phi(0,h_\infty)=0$. This is reasonable since the equilibria are contained in the linear phase manifold $\mathcal{P}\mathcal{M}_0$.

Let $(u_0,h_0)=(\tilde{u}_0,\tilde{h}_0)+(\phi(\tilde{u}_0,\tilde{h}_0),0)\in \mathcal{P}\mathcal{M}$ and let $(u,h,\pi,\eta)$ be the solution of \eqref{nonlinstab1} to this initial value on some interval $[0,a]$. With the help of the map $\phi$ we want to derive a decomposition for $(u,h)$. To be precise we want to write
	$$(u,h)=(u_\infty,h_\infty)+(\tilde{u},\tilde{h})+(\bar{u},\bar{h}),$$
where $(\tilde{u},\tilde{h})(t)\in \mathcal{P}\mathcal{M}_0$ for all $t\in [0,a]$ and $(u_\infty,h_\infty,\pi_\infty,\eta_\infty)$ is an equilibrium of \eqref{nonlinstab1}. Consider the two coupled systems
\begin{align}
\begin{split}\label{decomp1}
\omega \bar{u} +\partial_t \bar{u}-\mu^\pm\Delta \bar{u}+\nabla \bar{\pi}&=F_u(u,\pi,h)\\
{\rm div}\,\bar{u}&=F_d(u,h)\\
-P_\Sigma[\![\mu(\nabla\bar{u}+\nabla\bar{u}^{\sf T})]\!]\nu_\Sigma &= G_\tau(u,h)\\
-([\![\mu^\pm(\nabla\bar{u}+\nabla\bar{u}^{\sf T})]\!]\nu_\Sigma|\nu_\Sigma)+[\![\bar{\pi}]\!]-\sigma\cA_\Sigma\bar{h}&=G_\nu(u,h)+G_\gamma(h)-G_\gamma(h_\infty)\\
[\![\bar{u}]\!] &=0\\
\bar{u}|_{\partial\Omega}&=0\\
\omega \bar{h}+\partial_t\bar{h}-\bar{u}\cdot \nu_\Sigma-[\![\partial_{\nu_{\Sigma}}\bar{\eta}]\!]&=G_h(u,\eta,h)\\
\Delta \bar{\eta}&=F_\eta(\eta,h)\\
\bar{\eta}|_{\Sigma}+\mathcal{A}_{\Sigma}\bar{h}&=G_\eta(h)-G_\eta(h_\infty)\\
\partial_\nu\bar{\eta}|_{\partial\Omega}&=0\\
\bar{u}(0)=\phi(\tilde{u}_0,\tilde{h}_0),\quad \bar{h}(0)&=0,
\end{split}
\end{align}
and
\begin{alignat}{2}\label{decomp2}
\partial_t \tilde{u}-\mu^\pm\Delta \tilde{u}+\nabla \tilde{\pi}&=\omega \bar{u}&\quad& t>0,\ x\in\Omega\backslash\Sigma\nonumber\\
{\rm div}\,\tilde{u}&= 0&\quad& t>0,\ x\in\Omega\backslash\Sigma\nonumber\\
-P_\Sigma[\![\mu^\pm(\nabla\tilde{u}+\nabla\tilde{u}^{\sf T})]\!]\nu_\Sigma &=0&\quad& t>0,\ x\in\Sigma\nonumber\\
-([\![\mu(\nabla\tilde{u}+\nabla\tilde{u}^{\sf T})]\!]\nu_\Sigma)\cdot\nu_\Sigma+[\![\tilde{\pi}]\!]-\sigma\cA_\Sigma\tilde{h}&=0&\quad& t>0,\ x\in\Sigma\nonumber\\
[\![\tilde{u}]\!] &=0 &\quad& t>0,\ x\in\Sigma\\
\tilde{u}&=0&\quad& t>0,\ x\in\partial\Omega\nonumber\\
\partial_t\tilde{h}-\tilde{u}\cdot \nu_\Sigma-[\![\partial_{\nu_{\Sigma}}\tilde{\eta}]\!]&
=\omega\bar{h}&\quad& t>0,\ x\in\Sigma\nonumber\\
\Delta \tilde{\eta}&=0&\quad& t>0,\ x\in\Sigma\nonumber\\
\tilde{\eta}|_{\Sigma}+\mathcal{A}_{\Sigma}\tilde{h}&=0&\quad& t>0,\ x\in\Sigma\nonumber\\
\partial_\nu\tilde{\eta}&=0&\quad& t>0,\ x\in\partial\Omega\nonumber
\end{alignat}
with initial values $\tilde{u}(0)=\tilde{u}_0$ and $\tilde{h}(0)=\tilde{h}_0-h_\infty$. Here $\pi=\pi_\infty+\tilde{\pi}+\bar{\pi}$ and $\eta=\eta_\infty+\tilde{\eta}+\bar{\eta}$. We recall that $u_\infty=0$ and $\pi_\infty,\eta_\infty$ are constants and it holds that
$$\eta_\infty=[\![\pi_\infty]\!]/\sigma=\mathcal{H}(h_\infty).$$

With the help of the operator $A$ introduced above, we may rewrite problem \eqref{decomp2} as
	\begin{equation}\label{decomp3}
	\dot{\tilde{z}}(t)+A\tilde{z}(t)=R(\bar{z})(t)\quad \text{for}\ t\in (0,T),\ \tilde{z}(0)=\tilde{z}_0-z_\infty,
	\end{equation}
where $\tilde{z}_0:=(\tilde{u}_0,\tilde{h}_0)$, $z_\infty:=(0,h_\infty)$, $\tilde{z}=(\tilde{u},\tilde{h})$ and $\bar{z}=(\bar{u},\bar{h})$. Here the mapping $R$ is given by
	$$R(\bar{z})=(\omega(I-\mathcal{T}_1)\bar{u},\omega\bar{h}).$$
Thanks to Proposition \ref{propertiesA} we have the decomposition $X_0=N(A)\oplus R(A)$. Let $P^c$ denote the spectral projection corresponding to $\sigma_c(A)=\{0\}$ and set $P^s=I-P^c$. Then $R(P^c)=N(A)$ and $R(P^s)=R(A)$. Following the lines of \cite[Remark 2.2 (b)]{PSZ09} we may parametrize the set of equilibria near 0 over $N(A)$ via a $C^2$ map $[\sx\mapsto\sx+\psi(\sx)]$ such that $\psi(0)=\psi'(0)=0$ and $R(\psi)\subset D(A_s)$, where $A_s=AP^s$. This is true, since the nonlinearities on the right side in \eqref{nonlinstab1} are bilinear and smooth.

For $z_\infty$ sufficiently close to 0 there exists $\sx_\infty$ such that $z_\infty:=\sx_\infty+\psi(\sx_\infty)$. Introducing the new variables $\sx :=P^c\tilde{z}$ and
$$\sy:=P^s\tilde{z}-\psi(\sx_\infty+P^c\tilde{z})+\psi(\sx_\infty)$$
we obtain from \eqref{decomp3} the so-called normal form
\begin{alignat}{2}\label{normeq}
\dot{\sx}&=T(\bar{z}),\quad &\sx(0)=&\sx_0-\sx_\infty,\nonumber\\
\dot{\sy}+A_s\sy&=S(\sx_\infty,\sx,\bar{z}),\quad &\sy(0)=&\sy_0,
\end{alignat}
where $T(\bar{z})=P^cR(\bar{z})$,
$$S(\sx_\infty,\sx,\bar{z})=P^sR(\bar{z})-A_s(\psi(\sx_\infty+\sx)-\psi(\sx_\infty))-\psi'(\sx_\infty+\sx)T(\bar{z}),$$
and $\sx_0:=P^c\tilde{z}_0$, $\sy_0:=P^s\tilde{z}_0-\psi(\sx_0)$ with $\tilde{z}_0=(\tilde{u}_0,\tilde{h}_0)$. Observe that $S(0)=S'(0)=0$ by the properties of the function $\psi$ and since $T(0)=0$.

Let
$$\mathbb{E}_1(\mathbb{R}_+):=H^1(\mathbb{R}_+;L_{2,\sigma}(\Omega))\cap L_2(\mathbb{R}_+;H^2(\Omega\backslash\Sigma)),$$
$$\mathbb{E}_2(\mathbb{R}_+):=W_p^1(\mathbb{R}_+;W_p^{1-1/p}(\Sigma))\cap L_p(\mathbb{R}_+;W_p^{4-1/p}(\Sigma)).$$
and let
$$\mathbb{E}_1(\mathbb{R}_+,\delta):=\{v\in L_2(\mathbb{R}_+;L_2(\Omega)):e^{\delta t} v\in \mathbb{E}_1(\mathbb{R}_+)\},$$
$$\mathbb{E}_2(\mathbb{R}_+,\delta):=\{v\in L_p(\mathbb{R}_+;L_p(\Omega)):e^{\delta t} v\in \mathbb{E}_2(\mathbb{R}_+)\},$$
where $\delta\in (0,\delta_0)$ and $\delta_0>0$ depends on the spectral bound on the operator $A_s$ (see Proposition \ref{propertiesA}). Clearly, $T:\mathbb{E}(\mathbb{R}_+,\delta)\to \mathbb{E}^c(\mathbb{R}_+,\delta),$ where
$$\mathbb{E}(\mathbb{R}_+,\delta):=\mathbb{E}_1(\mathbb{R}_+,\delta)\times \mathbb{E}_2(\mathbb{R}_+,\delta)$$
and $\mathbb{E}^c(\mathbb{R}_+,\delta):=P^c\mathbb{E}(\mathbb{R}_+,\delta)$.
For given $(\sx_0,\sy_0,\bar{z})$ we want to solve \eqref{normeq} for $(\sx,\sy,\sx_\infty)$. First, for given $(\sx_0,\bar{z})\in X_0^c\times \mathbb{E}(\mathbb{R}_+,\delta)$ with $X_0^c:=P^cX_0$ we define
$$\sx_\infty:=\sx_0+\int_0^\infty T(\bar{z})(s)ds=:K_1(\sx_0,\bar{z})\in X_0^c.$$
Then
$$\sx(t):=-\int_t^\infty T(\bar{z})(s)ds=:K_2(\bar{z})$$
solves the first differential equation in \eqref{normeq} and
$$\sx(0)=-\int_0^\infty T(\bar{z})(s)ds=\sx_0-\sx_\infty.$$
Observe that by Young's inequality we have
$$\sx\in P^c[H^1(\mathbb{R}_+,\delta;L_{2,\sigma}(\Omega))\times W_p^1(\mathbb{R}_+,\delta;W_p^{1-1/p}(\Sigma))].$$
These exponentially weighted function spaces are defined in exactly the same way as $\mathbb{E}_j(\mathbb{R}_+,\delta)$.

Substituting the expressions for $\sx_\infty$ and $\sx$ into the function $S$, we obtain
$$\dot{\sy}+A_s\sy=S_1(\sx_0,\bar{z}),\quad \sy(0)=\sy_0,$$
where
$$S_1(\sx_0,\bar{z}):=S(K_1(\sx_0,\bar{z}),K_2(\bar{z}),\bar{z}),$$
and $\sy_0\in X_0^s\cap\mathcal{P}\mathcal{M}_0$. If one takes into account that the first component of $\psi$ is identically zero, it follows easily from the definition of $S$ and the smoothness of $\psi$ that
$$S_1(\sx_0,\bar{z}):X_0^c\times\mathbb{E}(\mathbb{R}_+,\delta)\to \mathbb{X}^s(\mathbb{R}_+,\delta)$$
where
$$\mathbb{X}^s(\mathbb{R}_+,\delta):=P^s[L_2(\mathbb{R}_+,\delta;L_{2,\sigma}(\Omega))\times L_p(\mathbb{R}_+,\delta;W_p^{1-1/p}(\Sigma))].$$
Here $X_0^s:=P^sX_0$. Since $\sigma(A_s)\subset\mathbb{C}_+$, we obtain for $\delta>0$ sufficiently small
$$\sy=\left(\frac{d}{dt}+A_s,{\rm tr}\right)^{-1}\left(S_1(\sx_0,\bar{z}),\sy_0\right)\in \mathbb{E}^s(\mathbb{R}_+,\delta),$$
where ${\rm tr}\ v:=v(0)$ and
$$\mathbb{E}^s(\mathbb{R}_+,\delta):=P^s\mathbb{E}(\mathbb{R}_+,\delta).$$
Here $\delta>0$ depends on the growth bound of the semigroup.
Putting things together, we see that
$$\tilde{z}=\tilde{\mathcal{G}}(\sx_0,\sy_0,\bar{z}):=\sx+\psi(\sx+\sx_\infty)-\psi(\sx_\infty)+\sy$$
and
$$z_\infty=\mathcal{G}_\infty(\sx_0,\sy_0,\bar{z}):=\sx_\infty+\psi(\sx_\infty).$$
We turn our attention to \eqref{decomp1}. Let $\mathbb{L}_\omega$ be the linear operator defined by the left side of \eqref{decomp1}. Then we can rewrite \eqref{decomp1} in the shorter form
\begin{equation}\label{eq:wbar}
\mathbb{L}_\omega\bar{w}=N(w_\infty+\tilde{w}+\bar{w})-N(w_\infty),
\end{equation}
with initial value $\bar{w}(0)=\bar{w}_0:=(\phi(\tilde{u}_0,\tilde{h}_0),0)$.

Here we have set $w_\infty=(u_\infty,h_\infty,0,0)$. Due to the first part of the proof, the nonlinearities on the right hand side of \eqref{decomp1} depend only on $(\sx_0,\sy_0,\bar{w})$, where $\bar{w}=(\bar{u},\bar{h},\bar{\pi},\bar{\eta})$ since $w_\infty=(\mathcal{G}_\infty(\sx_0,\sy_0,\bar{u},\bar{h}),0,0)$ and since there exists a function $\tilde{\mathcal{H}}$ such that
$$\tilde{w}=(\tilde{u},\tilde{h},\tilde{\pi},\tilde{\eta})=
\tilde{\mathcal{H}}(\sx_0,\sy_0,\bar{u},\bar{h}).$$
This follows from the considerations above, as $(\tilde{\pi},\tilde{\eta})$ can be written in terms of $(\bar{u},\bar{h})$ and $(\tilde{u},\tilde{h})=\tilde{z}=\tilde{\mathcal{G}}(\sx_0,\sy_0,\bar{z})$. Moreover, the right hand sides in \eqref{decomp1} do not depend on $(\pi_\infty,\eta_\infty)$, since these quantities are constant.

In order to solve \eqref{decomp1} we define
$$M(\sx_0,\sy_0,\bar{w}):=N(w_\infty+\tilde{w}+{\rm ext}_\delta[(\phi(\tilde{u}_0,\tilde{h}_0),0)-(\bar{u}(0),\bar{h}(0))]+\bar{w})-N(w_\infty),$$
where
$${\rm ext}_\delta:\{H_0^1(\Omega)^n\cap L_{2,\sigma}(\Omega)\}\times W_p^{4-4/p}(\Sigma)\to \mathbb{E}(\mathbb{R}_+,\delta)\times\{0\}\times\{0\},$$
such that $({\rm ext}_\delta z)(0)=(z_1,z_2,0,0)$ with $z=(z_1,z_2)$. We want to solve the equation
\begin{equation}\label{impl1}
\mathbb{L}_\omega\bar{w}=M(\sx_0,\sy_0,\bar{w}),\quad (\bar{w}_1,\bar{w}_2)(0)=(\phi(\tilde{u}_0,\tilde{h}_0),0),
\end{equation}
by the implicit function theorem. Let
$$\bar{\mathbb{E}}(\mathbb{R}_+,\delta):=\mathbb{E}(\mathbb{R}_+,\delta)\times L_2(\mathbb{R}_+,\delta;H_{(0)}^1(\Omega\backslash\Sigma))\times L_p(\mathbb{R}_+,\delta;W_p^2(\Omega\backslash\Sigma))$$
and define
$$K(\sx_0,\sy_0,\bar{w}):=\bar{w}-(\mathbb{L}_\omega,\operatorname{tr})^{-1}(M(\sx_0,\sy_0,\bar{w}),
(\phi(\tilde{u}_0,\tilde{h}_0),0)).$$
The mapping $K:\mathbb{B}(r,\delta)\to\bar{\mathbb{E}}(\mathbb{R}_+,\delta)$ is well defined, provided that $\omega>0$ is sufficiently large since $(M(\sx_0,\sy_0,\bar{w}),(\phi(\tilde{u}_0,\tilde{h}_0),0))$ satisfies all relevant compatibility conditions at $t=0$. To be precise, we have $\Jump{\phi(\tilde{u}_0,\tilde{h}_0)}=0$, $\phi(\tilde{u}_0,\tilde{h}_0)|_{\partial\Omega}=0$ as well as
$$\Div \phi(\tilde{u}_0,\tilde{h}_0)=\Div (\tilde{u}_0+\phi(\tilde{u}_0,\tilde{h}_0))=F_d(\tilde{u}_0+\phi(\tilde{u}_0,\tilde{h}_0),\tilde{h}_0),$$
since $\Div\tilde{u}_0=0$. Here we have set
\begin{multline*}
\mathbb{B}(r,\delta):=\\
\{(\sx_0,\sy_0,\bar{w})\in X_0^c\times(X_0^s\cap\mathcal{P}\mathcal{M}_0)\times\bar{\mathbb{E}}(\mathbb{R}_+,\delta):
\|(\sx_0,\sy_0,\bar{w})\|_{[\mathcal{P}\mathcal{M}_0]^2\times\bar{\mathbb{E}}(\mathbb{R}_+,\delta)}\le r\},
\end{multline*}
where $r>0$ is sufficiently small.

Note that $M(0,0,0)=M'(0,0,0)=0$ since $\phi(0,0)=\phi'(0,0)=0$. Therefore the implicit function theorem yields a ball
$$B(0,\rho)\subset X_0^c\times (X_0^s\cap \mathcal{P}\mathcal{M}_0)$$
and a unique solution $\bar{w}=\Phi(\sx_0,\sy_0)$ of \eqref{impl1}, where $\Phi\in C^1(B(0,\rho))$. Note that by construction, $\bar{w}$ is a solution of \eqref{decomp1}.

Finally this shows that $(\tilde{u}(t),\tilde{h}(t))$ as well as $(\bar{u}(t),\bar{h}(t))$ converge in $H_0^1(\Omega)^n\times W_p^{4-4/p}(\Sigma)$ to zero as $t$ tends to infinity at an exponential rate.

Therefore $(u(t),h(t))\to (u_\infty,h_\infty)$ in $H_0^1(\Omega)^n\times W_p^{4-4/p}(\Sigma)$ as $t\to\infty$, where the equilibrium $(u_\infty,h_\infty)$ is determined by $(u_0,h_0)$.
\end{proof}

\appendix
\section{Maximal Regularity for the linear Stokes System}

For the following let $\Omega_0^+$, $\Omega$ be bounded domains with $C^3$-boundary such that $\Gamma_0:=\partial\Omega_0^+\subset \Omega$ and let $\Omega_0^-=\Omega\setminus \overline{\Omega_0^+}$. Recall that $H^1_{(0)}(\Omega)=H^1(\Omega)\cap L_{2,(0)}(\Omega)$ and $H^{-1}_{(0)}(\Omega)= H^1_{(0)}(\Omega)$.
 
In this appendix we consider the unique solvability of the system
\begin{alignat}{2}\label{eq:InstatStokes1}
  \partial_t u -\mu^\pm \Delta u+\nabla q &= f &\quad& \text{in} \ \Omega^\pm_0\times (0,T),\\ 
  \Div u &= g &\quad& \text{in} \ \Omega^\pm_0\times (0,T),\\
  \Jump{u} &= 0 &\quad& \text{on} \ \Gamma_0\times (0,T),\\
\Jump{\nu_{\Gamma_0} \cdot T(u,q)} &= a&\quad& \text{on} \ \Gamma_0\times (0,T)=:\Gamma_{0,T},\\
u|_{\partial\Omega_0} &= 0 &&\text{on}\ \partial\Omega_0\times (0,T),
\\\label{eq:InstatStokes5}
u|_{t=0} &= v_0  &\quad& \text{on} \ \Omega_0,
\end{alignat}
where $T(u,q)= \mu^\pm Du -qI$ in $\Omega_0^\pm$.


\begin{theorem}\label{thm:MaxReg}
  Let $0<T\leq T_0<\infty$, $n\geq 2$, and $\Omega\subset \Rn$ be a bounded domain with $C^3$-boundary. For every $v_0\in H^1_{0}(\Omega)^n, f\in L_2(Q_T)^n, g\in L_2(0,T;H^1(\Omega\setminus \Gamma_0))$ with $g\in H^1(0,T;H^{-1}_{(0)}(\Omega))$, 
$$
a\in H^{\frac14}(0,T;L_2(\Gamma_0))\cap L_2(0,T;H^{\frac12}(\Gamma_0))=: H^{\frac14,\frac12}(\Gamma_{0,T})
$$ 
such that
  \begin{equation}
    \label{eq:CompCond1}
    \Div v_0= g|_{t=0},\quad \int_{\Omega} g(t,x)\sd x =0\quad \text{for almost all}\ t\in (0,T) 
  \end{equation}
 the system \eqref{eq:InstatStokes1}-\eqref{eq:InstatStokes5} has a unique solution 
 \begin{equation*}
   u\in H^1(0,T;L_{2}(\Omega)^n)\cap L_\infty(0,T;H^1_{0}(\Omega)^n)\cap L_2(0,T;H^2(\Omega\setminus \Gamma_0)^n)
 \end{equation*}
 Moreover, there is some constant $C$ independent of $T\in (0,T_0]$, $u,f,g,a,v_0$ such that
  \begin{eqnarray*}
    \lefteqn{\|(\partial_t u,\nabla u)\|_{L_2(J;L_2(\Omega))}+ \sum_{\pm}\|(\nabla^2 u,\nabla q)\|_{L_2(J;L_2(\Omega_0^\pm))}}\\
 &\leq& C_q\left(\|(f,\nabla g)\|_{L_2(Q_T)}+\|\partial_t g\|_{L_2(J;H^{-1}_0(\Omega))}+\|a\|_{H^{\frac14,\frac12}(\Gamma_{0,T})}+ \|v_0\|_{H^1(\Omega)}\right)
  \end{eqnarray*}
where $J=[0,T]$. 
\end{theorem}
\begin{remark}
  The result follows from a result announced by Shimizu~\cite{ShmizuBanachCenterPubl}, where a general $L^q$-theory is discussed. In the case $q=2$, the proof is much simpler since Hilbert-space methods are available and the result basically follows from the resolvent estimate proved by Shibata and Shimizu in \cite{ShibataShimizuTwoPhase}. For the convenience of the reader we include a proof.
\end{remark}
\begin{proof*}{of Theorem~\ref{thm:MaxReg}} 
  First we consider the case $g=a=v_0=0$. We can assume without loss of generality that $f\in L_2(0,T;L_{2,\sigma}(\Omega))$. Otherwise we replace $f$ by $P_\sigma f$ and $q$ by $q-q_1$, where $\nabla q_1=f-P_\sigma f$. Then (\ref{eq:InstatStokes1})-(\ref{eq:InstatStokes5}) are equivalent to the abstract evolution equation
  \begin{eqnarray}\label{eq:AbstractEvol1}
    \frac{d}{dt} u(t) + Au(t) &=& f(t),\qquad t\in (0,T),\\\label{eq:AbstractEvol2}
u(0)&=&0,
  \end{eqnarray}
  where $A\colon \mathcal{D}(A)\to L_{2,\sigma}(\Omega)$ with
  \begin{eqnarray*}
    Au|_{\Omega^\pm}&=& -\nu^\pm \Delta u+\nabla q\\
\mathcal{D}(A) &=& \left\{u\in H^1_0(\Omega)^n\cap L_{2,\sigma}(\Omega): \nabla^2 u|_{\Omega_0^\pm}\in L^2(\Omega_0^\pm), \Jump{2\nu\cdot \mu^\pm Dv}_\tau =0\right\}
  \end{eqnarray*}
where $q\in L_{2,(0)}(\Omega)$ with $\nabla q|_{\Omega_0^\pm}\in L_2(\Omega_0^\pm)^n$ is uniquely defined by
\begin{alignat*}{2}
  \Delta q &=0&\qquad& \text{in}\ \Omega^\pm_0,\\~
  \Jump{q} &= \Jump{2\mu^\pm\partial_\nu v_\nu} &&\text{on}\ \Gamma_0,\\
  \partial_\nu q|_{\partial\Omega} &=\nu\cdot \mu^- \Delta u|_{\partial\Omega} &&\text{on}\ \partial\Omega.
\end{alignat*}
Because of \cite[Theorem~1.1]{ShibataShimizuTwoPhase}, $A$ is a generator of an exponentially decaying analytic semi-group and the graph norm $\|u\|_{\mathcal{D}(A)}$ is equivalent to
\begin{equation*}
  \|u\|_{H^1(\Omega)} +\sum_{\pm}\|\nabla^2u\|_{L_2(\Omega_0^\pm)}.
\end{equation*}
 Since $L_{2,\sigma}(\Omega)$ is a Hilbert space, for every $f\in L_2(0,T;L_{2,\sigma})$ there is a unique $u\in H^1(0,T;L_{2,\sigma})\cap L_2(0,T;\mathcal{D}(A)) $ solving (\ref{eq:AbstractEvol1})-(\ref{eq:AbstractEvol2}) and
 \begin{equation*}
   \left\|\frac{d}{dt}u\right\|_{L_2(0,T;L_{2,\sigma})}+   \left\|A u\right\|_{L_2(0,T;L_{2,\sigma})}\leq C\|f\|_{L_2(0,T;L_{2,\sigma})}
 \end{equation*}
with some $C>0$ independent of $T\in (0,\infty]$. In the case $T=\infty$ this statement follows from \cite{DeSimonMaxReg} or
\cite[Theorem 4.4]{DenkHieberPruess}, part ``(ii) implies (i)'', where we note that $\mathcal{R}$-boundedness of an operator family on a Hilbert space coincides with uniform boundedness, cf. \cite[Section 3.1]{DenkHieberPruess}. The case $T<\infty$ follows from the latter case by extending $f\colon (0,T)\to H$ by zero to $\tilde{f}\colon (0,\infty)\to H$. 
This proves the theorem in the case $g=a=v_0=0$.

The general case can be reduced to the latter case as follows: 
First we reduce to the case $(f,g,v_0)|_{\Omega_0^+}\equiv 0$. To this end let
\begin{equation*}
v^+\in H^1(0,T;L_2(\Omega_0^+)^n)\cap L_2(0,T;H^2(\Omega_0^+)^n)
,\quad q^+\in L^2(0,T;H^1(\Omega_0^+))  
\end{equation*}
be the solution of
\begin{alignat*}{2}
  \partial_t v^+ -\mu^+ \Delta v^++\nabla q^+ &= f|_{\Omega_0^+} &\quad& \text{in} \ \Omega^+_0\times (0,T),\\ 
  \Div v^+ &= g|_{\Omega_0^+} &\quad& \text{in} \ \Omega^+_0\times (0,T),
\\
\nu_{\Gamma_0}\cdot(2\mu^+ Dv^+-q^+) &= 0 && \text{on} \ \Gamma_0\times (0,T),\\
v^+|_{t=0} &= v_0|_{\Omega_0^+}  &\quad& \text{in} \ \Omega_0^+.
\end{alignat*}
The existence of such a $v^+$ follows from well known results for the instationary Stokes system with Neumann boundary conditions, cf. e.g. \cite{SolonnikovProceedings}.
Moreover, there is some constant $C>0$ such that for every $0<T\leq \infty$
\begin{eqnarray*}
    \lefteqn{\|(\partial_t v^+,\nabla v^+,\nabla^2 v^+,\nabla q^+)\|_{L_2(J;L_2(\Omega^+_0))}}\\
 &\leq& C_q\left(\|(f,\nabla g)\|_{L_2(J\times\Omega_0^+)}+\|\partial_t g\|_{L_2(J;H^{-1}(\Omega_0^+))}+\|v_0\|_{H^1(\Omega_0^+)}\right).
\end{eqnarray*} 
Now we extend $v^+$ and $q^+$ to some functions 
\begin{equation*}
 \tilde{v}^+\in L_2(0,T;H^2(\Omega_0))^n\cap H^1(0,T;L_2(\Omega_0))^n,\quad \tilde{q}^+\in L^2(0,T;H^1(\Omega_0^+)) 
\end{equation*}
 satisfying an analoguous estimate as before.
Now subtracting $(\tilde{v}^+,\tilde{q}^+)$ from $(u,q)$ we  reduce to the case $(f,g,v_0)|_{\Omega_0^+}\equiv 0$. Next we observe that 
$$g|_{\Omega_0^-}\in H^1(0,T;H^{-1}_{(0)}(\Omega_0^-))$$ 
because of 
\begin{equation*}
  \int_{\Omega_0^+} g(x,t)\varphi (x)\sd x = \int_{\Omega_0} g(x,t)\tilde{\varphi}(x)\sd x
\end{equation*}
for every $\varphi \in H^1_{(0)}(\Omega_0^+)$, where $\tilde{\varphi}\in H^1_{(0)}(\Omega_0)$ is an arbitrary extension of $\varphi$ to $\Omega_0$. Hence there are some
\begin{equation*}
v^-\in H^1(0,T;L_2(\Omega_0^-)^n)\cap L_2(0,T;H^2(\Omega_0^-)^n),\quad q^-\in L^2(0,T;H^1(\Omega_0^-))  
\end{equation*}
solving
\begin{alignat*}{2}
  \partial_t v^- -\mu^- \Delta v^-+\nabla q^- &= f|_{\Omega_0^-} &\quad& \text{in} \ \Omega^-_0\times (0,T),\\ 
  \Div v^- &= g|_{\Omega_0^-} &\quad& \text{in} \ \Omega^-_0\times (0,T),
\\
v^- &= 0 && \text{on} \ \Gamma_0\times (0,T),\\
v^-|_{t=0} &= v_0|_{\Omega_0^-}  &\quad& \text{in} \ \Omega_0^-.
\end{alignat*}
Existence of such a solution together with analoguous estimates as for $(v^+,q^+)$ follows e.g. from \cite{SolonnikovProceedings,FroehlichMaxReg,HieberEtAlStokesMaxRegShort,GigaSohrAC, SolonnikovStokesMaxReg}. Now extending $v^-$ and $q^-$ by zero to $\Omega_0$ and subtracting the extensions from $(u,p)$ we can reduce to the case $(f,g,v_0)\equiv 0$.

In order to reduce to the case, where also $a_\tau\equiv 0$, we construct some $A\in H^1(0,T;L_2(\Omega_0^+))\cap L_2(0,T;H^2(\Omega_0^+))$ such that
\begin{equation*}
  \|(A,\partial_t A, \nabla A, \nabla^2 A)\|_{L_2(Q_T^+)}\leq C\|a\|_{H^{\frac14,\frac12}(\Gamma_{0,T})}
\end{equation*}
and
\begin{equation*}
  A|_{t=0}=A|_{\Gamma_0}=0,\quad (\nu_{\Gamma_0} \cdot 2\mu^+ DA)_\tau|_{\Gamma_0}=a_\tau, \quad \Div A=0\quad \text{in}\ \Omega_0.
\end{equation*}
This can be done as follows: Choose some 
$$
\tilde{A}\in H^1(0,T;L_2(\Omega_0^+)^n)\cap L_2(0,T;H^2(\Omega_0^+)^n)
$$ 
such that
\begin{equation*}
  \|(\tilde{A},\partial_t \tilde{A}, \nabla \tilde{A}, \nabla^2 \tilde{A})\|_{L_2(\Omega_0^+\times (0,T))}\leq C\|a\|_{H^{\frac14,\frac12}(\Gamma_{0,T})}
\end{equation*}
and
\begin{equation*}
  \tilde{A}|_{\Gamma_0}=\tilde{A}|_{t=0}=0,\quad (\nu_{\Gamma_0} \cdot 2\mu^+ D\tilde{A})_\tau|_{\Gamma_0}=a_\tau,\quad \Div \tilde{A}|_{\Gamma_0}=0.
\end{equation*}
The existence of such an $\tilde{A}$ e.g. follows from \cite[Lemma~2.4]{SolonnikovProceedings}. Moreover, $C>0$ in the estimate above can be chosen independently of $0<T\leq T_0$ for any $T_0>0$.
Since $\Div \tilde{A}|_{\Gamma_0}=0$, $\Div \tilde{A}\in H^1_0(\Omega_0^+)$ and we can apply the Bogovski operator $B$, cf. e.g. \cite{Galdi}, to $\Div \tilde{A}$. Then we obtain $B(\Div \tilde{A})\in L^2(J;H^2_0(\Omega_0^+)\cap L^2_{(0)}(\Omega_0^+))$ and
\begin{equation*}
  \|B(\Div \tilde{A})\|_{L_2(J;H^2(\Omega_0^+))}\leq C\|\tilde{A}\|_{L_2(J;H^2(\Omega_0^+))}.
\end{equation*}
Moreover, due to \cite[Theorem~2.5]{GeissertHeckHieberBogovski} we also have
\begin{equation*}
  \|B(\Div \tilde{A})\|_{H^1(J;L_2(\Omega_0^+))}\leq C\|\Div \tilde{A}\|_{H^1(J;H^{-1}_{(0)}(\Omega_0^+))} \leq C'\|\tilde{A}\|_{H^1(J;L^2(\Omega_0^+))}. 
\end{equation*}
Since the Bogovski operator is independent of time, the latter constant can be chosen independently of $0<T\leq T_0$ for any $T_0>0$.
Altogether, we obtain that $A:= \tilde{A}- B(\Div \tilde{A})$ has the properties stated above. Replacing  $u$ by $u-A\chi_{\Omega_0^+}$, we can finally reduce to the case $v_0\equiv g\equiv a_\tau\equiv 0$.
Finally, we can also reduce to the case $a_\nu \equiv 0$ by substracting a suitable extension of $a_\nu$ from the pressure $q$.
\end{proof*}

\noindent
{\bf Acknowledgments:} We are grateful to the anonymous referees and Stefan Schau\-beck for careful reading previous versions of this work and many comments, which improved the paper. Moreover, the authors acknowledge support from the German
Science Foundation through Grant Nos. AB285/3-1 and AB285/4-1.

M.W. would like to express his thanks to Gieri Simonett for inspiring discussions concerning the proof of the stability result.

\bibliographystyle{amsplain}

\def\cprime{$'$} 
\providecommand{\bysame}{\leavevmode\hbox to3em{\hrulefill}\thinspace}
\providecommand{\MR}{\relax\ifhmode\unskip\space\fi MR }
\providecommand{\MRhref}[2]{%
  \href{http://www.ams.org/mathscinet-getitem?mr=#1}{#2}
}
\providecommand{\href}[2]{#2}

\end{document}